\documentclass[preprint,12pt]{elsarticle}

\usepackage{amssymb}
\usepackage{amsmath,amssymb,amsthm,bm,float,url,color,hyperref}
\usepackage[utf8]{inputenc}
\usepackage{t1enc, lmodern}
\usepackage{graphicx}

\usepackage[margin=1in]{geometry}
\usepackage{algorithm, algorithmicx, algpseudocode}

\newcommand{\rset}{\mathbb{R}}

\newcommand{\nset}{\mathbb{N}}
\newcommand{\nl}{\nolimits}

\newcommand{\ind}{\mathbf{1}}

\newcommand{\e}{\mathbb{E}}
\newcommand{\V}{\mathbb{V}}
\newcommand{\p}{\mathbb{P}}
\newcommand{\D}{\mathbb{D}}

\newcommand{\lp}{\mathrm{L}}
\newcommand{\spp}{\mathrm{S}}
\newcommand{\hpp}{\mathrm{H}}
\newcommand{\m}{\mathcal}
\newcommand{\cf}{\mathcal{F}}
\newcommand{\cc}{\mathcal{C}}

\def\norm#1{\|#1\|}
\def\bY{\bm{Y}}
\def\bZ{\bm{Z}}
\def\bU{\bm{U}}
\def\bG{\bm{G}}
\def\bQ{\bm{Q}}

\def\ov{\overline}

\setcitestyle{square}
\theoremstyle{plain}
\newtheorem{thm}{Theorem}[section]
\newtheorem{lemme}[thm]{Lemma}

\newtheorem{prop}[thm]{Proposition}

\newtheorem{hypo}[thm]{Hypothesis}

\theoremstyle{definition}
\newtheorem{definition}[thm]{Definition}
\newtheorem{ex}[thm]{Example}
\theoremstyle{plain}
\newtheorem{rem}[thm]{Remark}

\journal{Stochastic Proc. Appl.}

\begin{document}

\begin{frontmatter}



\title{Simulation of BSDEs with jumps by Wiener Chaos Expansion}



\author[label1]{Christel Geiss\corref{cor1}}  \author[label2]{Céline Labart\corref{cor2}}

\address[label1]{Department of Mathematics and Statistics,
P.O.Box 35 (MaD),FI-40014 University of Jyväskylä, Finland}
\address[label2]{LAMA - Universit\'e de Savoie,
Campus Scientifique,
73376 Le Bourget du Lac, France}

\cortext[cor1]{christel.geiss@jyu.fi \tiny(corresponding author)}
\cortext[cor2]{celine.labart@univ-savoie.fr}
\begin{abstract}
We present an algorithm to solve BSDEs with jumps based on
   Wiener Chaos Expansion and Picard's iterations. This paper extends the
   results given in \cite{BL_14} to the case of BSDEs with jumps. We get a
  forward scheme where the conditional expectations are easily computed thanks
   to chaos decomposition formulas. Concerning the error, we derive explicit
  bounds with respect to the number of chaos, the discretization time step and
   the number of Monte Carlo simulations. We also present numerical
   experiments. We obtain very encouraging results in terms of speed and
  accuracy.
\end{abstract}


\end{frontmatter}

\section{Introduction}

\label{sec:introduction} In this paper we are interested in the numerical
approximation of solutions $(Y,Z,U)$ to backward stochastic differential
equations (BSDEs in the sequel) with jumps of the following form
\begin{equation}\label{eq:main}  Y_t=\xi+\int_t^T f(s,Y_s,Z_s,U_s)\, ds -
  \int_t^T Z_s dB_s-\int_{]t,T]} U_s d\tilde{N}_s,\quad 0\leq t\leq T,
\end{equation} 
where $B$ is a $1$-dimensional standard Brownian motion and $\tilde{N}$ is a
compensated Poisson process independent from B, i.e.
$\tilde{N}_t:=N_t-\kappa t$ and $\{N_t\}_{t \ge 0}$ is a Poisson process with intensity
$\kappa>0.$ The terminal condition $\xi$ is a real-valued $\m F_T$--measurable
random variable where $\{\m F_t\}_{0\leq t\leq T}$ stands for the augmented
natural filtration associated with $B$ and $N$. Under standard Lipschitz
assumptions on the driver $f$, the existence and uniqueness of the solution
have been stated by Tang and Li \cite{TL_94},
generalizing the seminal paper of Pardoux and Peng \cite{PP_92}. \smallskip

The main objective of this paper is to propose a numerical method to
approximate the solution $(Y,Z,U)$ of \eqref{eq:main}. In the no-jump case,
there exist several methods to simulate $(Y,Z)$. The most popular one is the
method based on the dynamic programming equation, introduced by Briand, Delyon and
Mémin \cite{BDM_02}. In the Markovian case, the rate of convergence of the
method has been studied by Zhang \cite{Zha_04} and Bouchard and Touzi
\cite{BT_04}. From a numerical point of view, the main difficulty in solving
BSDEs is to compute conditional expectations. Different approaches have been
proposed: Malliavin calculus \cite{BT_04}, regression methods \cite{GLW_05}
and quantization techniques \cite{BP_03}. In the general case (i.e. for a
terminal condition which is not necessarily Markovian), Briand and Labart
\cite{BL_14} have proposed a forward scheme based on Wiener chaos expansion
and Picard's iterations. Thanks to the chaos decomposition formulas,
conditional expectations are easily computed, which leads to an efficient,
fully implementable scheme. In case of BSDEs driven by a Poisson random
measure, Bouchard and Elie \cite{BE_08} have proposed a scheme based on the
dynamic programming equation and studied the rate of convergence of the method
when the terminal condition is given by $\xi=g(X_T)$, where $g$ is a Lipschitz
function and $X$ is a forward process. More recently, Geiss and Steinicke
\cite{GS_13} have extended this result to the case of a terminal condition
which may be a Borel function of finitely many increments of the Lévy forward
process $X$ which is not necessarily Lipschitz but only satisfies a fractional
smoothness condition. In the case of jumps driven by a compensated Poisson
process, Lejay, Mordecki and Torres \cite{LMT_14} have developed a fully
implementable scheme based on a random binomial tree, following the
approach proposed by Briand, Delyon and Mémin \cite{BDM_01}.\smallskip

In this paper, we extend the algorithm based on Picard's iterations and Wiener
chaos expansion introduced in \cite{BL_14} to the case of BSDEs with
jumps. Our starting point is the use of Picard's iterations: $(Y^0,Z^0,U^0)=(0,0,0)$
and for $q\in\nset$,
\begin{equation*} Y^{q+1}_t = \xi + \int_t^T f\left(s,Y^q_s,Z^q_s,U^q_s\right) ds -
  \int_t^T Z^{q+1}_s\cdot dB_s-\int_{]t,T]} U^{q+1}_s d\tilde{N}_s, \quad 0\leq t\leq T.
\end{equation*}
Writing this Picard scheme in a forward way gives
\begin{align*}
	Y^{q+1}_t & = \e\left(\xi + \int_0^T f\left(s,Y^q_s,Z^q_s,U^q_s\right) ds \:\Big|\: \m F_t\right) - \int_0^t f\left(s,Y^q_s,Z^q_s,U^q_s\right) ds, \\
	Z^{q+1}_t & =\e\left(  D^{(0)}_t Y^{q+1}_t\:\Big|\: \m F_{t^-}\right) =  \e\left( D^{(0)}_t \left (\xi + \int_0^T
      f\left(s,Y^q_s,Z^q_s,U^q_s\right) ds \right ) \:\Big|\: \m F_{t^-}\right),\\
    U^{q+1}_t & = \e\left(  D^{(1)}_t Y^{q+1}_t\:\Big|\: \m F_{t^-}\right) =  \e\left( D^{(1)}_t\left ( \xi + \int_0^T
      f\left(s,Y^q_s,Z^q_s,U^q_s\right) ds \right)\:\Big|\: \m F_{t^-}\right),\\
\end{align*}
where $D^{(0)}_t X$ (resp. $D^{(1)}_t X$) stands for the Malliavin derivative of the random variable
$X$ with respect to the Brownian motion (resp. w.r.t. the Poisson process).\smallskip

In order to compute the previous conditional expectation, we use a Wiener chaos expansion of the random variable
\begin{equation*}
F^{q}=\xi + \int_0^T f\left(s,Y^q_s,Z^q_s,U^q_s\right) ds .
\end{equation*}
More precisely, we use the following orthogonal decomposition of the random
variable $F^q$ (see Proposition \ref{chaos-decomposition})
\begin{equation*} F^q = \e\left[F^q\right] + \sum_{k=1}^\infty \sum_{l=0}^k \, \sum_{{\bf k}_l \in \nset^l  }  \,\,  \sum_{{\bf j}_{k-l} \in \nset^{k-l}}   
  d_{{\bf k}_l,  \, {\bf j}_{k-l}} L^{0,...,0}_l(\tilde e[k_1,\ldots,k_l]) \, L^{1,...,1}_{k-l}(\tilde e[j_1,\ldots,j_{k-l}]).
\end{equation*}
where $L^{0, \cdots,0}_m(g)$ (resp. $L^{1, \cdots,1}_{m}(g)$) denotes the
iterated integral of order $m$ of $g$ w.r.t.~the Brownian motion
(resp. w.r.t.~the compensated Poisson process), $(\tilde e[k_1,\ldots,k_m])_{k_m \in \nset}$ is an
orthogonal basis of $(\tilde{L}^2)^{\otimes m}([0,T])$, the subspace of
symmetric functions from $(L^2)^{\otimes m}([0,T])$. The sequence of coefficients $\{d_{{\bf k}_l, \,{\bf
      j}_{k-l}}\}_{{\bf k}_l \in \nset^l, \, {\bf j}_{k-l} \in \nset^{k-l}}$ 
   ensues from the Wiener chaos decomposition of $F^q$. \smallskip

 The point to get an implementable scheme is that we only keep a finite number of terms in this
  expansion: we use a finite number of chaos and we choose a finite number of
  functions $\{e_1,\cdots,e_N\}$ to build $\{\tilde{e}[k_1,\cdots,k_m]\}_{k_m \in
    \{1,\cdots,N\}}$. More precisely, if we choose
  $e_i:=\frac{1}{\sqrt{h}}\ind_{]\overline{t}_{i-1},\ov{t}_i]}$ where
    $\overline{t}_i=ih$ and $h:=\frac{T}{N}$, we obtain
    \begin{align*}
      F^q \sim \e\left[F^q\right]  +\sum_{k=1}^p \sum_{|n|=k} d^n_k \prod_{i=1}^N
      K_{n^B_i}\left(\frac{B_{\overline{t}_{i}}-B_{\overline{t}_{i-1}}}{\sqrt{h}}\right)C_{n^P_i}(N_{\overline{t}_{i}}-N_{\overline{t}_{i-1}},\kappa h),
    \end{align*} where $K_i$ (resp. $C_i$) denotes the Hermite
    (resp. Charlier) polynomial of degree $i$,
    $n=(n^B_1,\cdots,n^B_N,n^P_1,$ $\cdots,n^P_N)$ is a vector of integers and
    $|n|=\sum_{i=1}^N (n^B_i+n^P_i)$. By using this approximation of $F^q$ we
    can easily compute $\e(F^q|\cf_t)$,  $\e( D^{(0)}_t F^q|\cf_{t^-})$ and
    $\e( D^{(1)}_t F^q|\cf_{t^-})$, which gives us
    $(Y^{q+1}_t,Z^{q+1}_t,U^{q+1}_t)$. To get a fully implementable algorithm,
    it remains to approximate $\e(F^q)$ and the coefficients $\{d^n_k\}_{n,k}$ by Monte Carlo.  \smallskip

    When extending  \cite{BL_14} to the jump case one realizes that the main difficulty lies in the
    fact that there is no hypercontractivity property in the Poisson chaos
    decomposition case. This property plays an important role in the proof of
    the  convergence in the Brownian case. To circumvent this problem, we
   exploit a recent result of Last, Penrose, Schulte and Thäle
    \cite{LPST_14}, which gives a formula to compute the
    expectation of products of Poisson multiple integrals, and the according 
result for the Brownian case from  Peccati and Taqqu \cite{peccati}. In fact, in  equation \eqref{expectation-product-formula} of Proposition
    \ref{E-of-chaos-products}  we get an  explicit expression for 
\[\e ( I_{n_1}(f_{n_1}) \cdots   I_{n_l}(f_{n_l}))\] 
in terms of a combinatoric sum of tensor products of the chaos kernels $f_{n_i}.$
Here $I_{n_i}(f_{n_i})$ denotes the multiple integral of order $n_i$
with respect to the process $B + \tilde N.$  By this  expression  one  gets the required estimates for the truncated chaos without the
hypercontractivity  property. 
   Therefore, to  prove the convergence of the method we may proceed  similarly to  \cite{BL_14}, and  split the error into four terms:
    \begin{itemize}
    \item the error due to Picard iterations
    \item the error due to the truncation onto  the  chaos up to order $p$
    \item the error due to the finite number of basis functions $\{e_1,\cdots,e_N\}$ for each chaos
    \item the error due to the Monte Carlo simulations to approximate the
      expectations appearing in the coefficients $\{d^n_k\}_{n,k}.$
    \end{itemize} 
\bigskip
    The  paper is organized as follows: Section \ref{sect:WCE}
    contains the notations and gives preliminary results, Section
    \ref{sect:numerical-scheme} describes the approximation procedure,
    Section \ref{sect:conv-results} states the convergence results and Section
    \ref{sect:implementation} presents the algorithm and some numerical
    examples. Some technical results are proved in the appendix.

  \subsection{Definitions and Notations}
Given a probability space $(\Omega,\cf,\p)$ we consider
\begin{itemize}
\item $\lp^p(\m F_T):=\lp^p(\Omega,\cf_T,\p)$, $p\in \nset^*= \nset \setminus \{0\}$, the space of
  all $\cf_T$-measurable random variables (r.v. in the following) $X:\Omega
  \longmapsto \rset$ satisfying $\norm{X}^p_p: = \e(|X|^p)< \infty$.
\item $\spp^p_T(\rset)$, $p \in \nset, p\ge 2$, the space of all càdlàg, adapted processes $\phi:\Omega\times [0,T] \longmapsto \rset$ such that $\norm{\phi}^p_{\spp^p}=\e
  (\sup_{t \in [0,T]} |\phi_t|^p)< \infty$.
\item $\hpp^p_T(\rset)$, $p \in \nset, p\ge 2$, the space of all predictable processes
  $\phi:\Omega\times [0,T] \longmapsto \rset$ such that $\norm{\phi}^p_{\hpp^p_T}=\e
  (\int_0^T |\phi_t|^p dt) < \infty$.
\item $\lp^2(0,T)$, the space of all square integrable functions on $[0,T]$.
\item $C^{k,l}$, the set of continuously differentiable functions $\phi:(t,x)
  \in [0,T] \times \rset^3$ with continuous derivatives w.r.t. $t$
  (resp. w.r.t. $x$) up to order $k$ (resp. up to order $l$).
\item $C^{k,l}_b$, the set of continuously differentiable functions
  $\phi:(t,x) \in [0,T] \times \rset^3$ with continuous and uniformly bounded
  derivatives w.r.t. $t$ (resp. w.r.t. $x$) up to order $k$ (resp. up to order
  $l$). The function $\phi$ is also bounded.
\item $\norm{\partial^j_{sp} f}^2_{\infty}$, the sum of the squared norms of
  the derivatives of $f([0,T] \times \rset^3, \rset)$ w.r.t. all the space
  variables $x$ which sum equals $j$ : $\norm{\partial^j_{sp}
    f}^2_{\infty}:=\sum_{|k|=j}
  \norm{\partial^{k_1}_{x_1} \partial^{k_2}_{x_2} \partial^{k_3}_{x_3}
    f}^2_{\infty}$, where $|k|=k_1+k_2+k_3$.
\item $C^{\infty}_p$, the set of smooth functions $f:\rset^n \longmapsto
  \rset$  ($n \ge 1$) with partial derivatives of polynomial growth.
\item $\norm{(\cdot,\cdot, \cdot)}^p_{\lp^p}$, $p\ge 1$, the norm on the space $\spp^p_T(\rset)\times
  \hpp^p_T(\rset)\times
  \hpp^p_T(\rset)$ defined by
  \begin{align}\label{norme_L2_YZ}
    \norm{(Y,Z,U)}^p_{\lp^p}:=
  \e( \sup_{t \in [0,T]} |Y_t|^p)+\int_0^T \e(|Z_t|^p) dt+\kappa \int_0^T \e(|U_t|^p) dt .
\end{align}
\end{itemize}

\begin{hypo} \label{hypo1} 
 We assume
\begin{itemize}
	\item the terminal condition $\xi$ belongs to $\lp^2(\m F_T)$;
	\item the generator $f \in C([0,T] \times \rset^3;\rset)$ is Lipschitz
      continuous  in space, uniformly in $t$: there exists a constant $L_f$
      such that
      \begin{equation*} |f(t,y_1,z_1,u_1)-f(t,y_2,z_2,u_2)| \leq
        L_f\left(|y_1-y_2|+|z_1-z_2|+|u_1-u_2|\right).
	\end{equation*}
  \end{itemize}
\end{hypo}

\begin{lemme}
  If Hypothesis \ref{hypo1} is satisfied and  $\xi \in  \D^{1,2}$ (defined below) we get from \cite[Theorem 3.4]{GS_13} that
  for a.e. $t \in [0,T]$
  \begin{align}\label{eq3}
    Z_t= \e[D^{(0)}_t Y_t| \m F_{t-}] ,\quad U_t= \e[ D^{(1)}_t Y_t| \m F_{t-}] \,\,\,\, \p-{\mbox a.s.}
  \end{align} where $D^{(0)}_t X$ stands for the Malliavin derivative
  w.r.t. the Brownian motion of the random variable $X$, and $D^{(1)}_t X$
  stands for the Malliavin derivative w.r.t. the Poisson process of the random
  variable $X$. Here  $ \e[ \cdot | \m F_{t-}] $ should be understood as  the predictable projection,
  and since the paths $s \mapsto D^{(i)}_t Y_s$ are a.s.  c\`adl\`ag  we define $D^{(i)}_t Y_t := \lim_{s \downarrow t} D^{(i)}_t Y_s$
  if the limit exists, and zero otherwise.
\end{lemme}

\section{Wiener Chaos Expansion}\label{sect:WCE}
\subsection{Notations and useful results}

\subsubsection{ Iterated integrals}
We refer to \cite{LSUV_02} and \cite{privault} for more
details on this section. Let us briefly recall the Wiener chaos
expansion in the case of a real-valued Brownian motion and an independent
Poisson process with intensity $\kappa >0.$

We define
\begin{align*}
  G_0(t)=B_t,\quad G_1(t)=N_t - \kappa t,
\end{align*}
and $L_k^{i_1,\cdots,i_k}(f)$ the iterated integral of $f$ with respect to $G_0$ and
$G_1$
\begin{align*}
  L_k^{i_1,\cdots,i_k}(f)=\int_0^T\left(\int_0^{t_k^-}\cdots \left(\int_0^{t_2^-}
      f(t_1,\ldots,t_k) dG_{i_1}(t_1)\right) \cdots dG_{i_{k-1}}(t_{k-1})\right)dG_{i_k}(t_k).
\end{align*}

We have the following chaotic representation property.
\begin{prop} \label{CRP-for-iterated} 
(\cite[Proposition 2.1]{LSUV_02})  For  $k\in \nset^*$  define
  \[\mathbf{i}_k:=(i_1,\ldots,i_k)  \in \{0,1\}^k.  \] 
  Any 
  $F   \in \lp^2(\m F_T)$ has a   unique  representation of the form
  \begin{align}\label{eq0}
    F=\e(F)+\sum_{k=1}^{\infty} \,\, \sum_{\mathbf{i}_k \in \{0,1\}^k }   L_k^{ {\mathbf{i}_k}}   (f_{\mathbf{i}_k}),
  \end{align}
  where $f_{\mathbf{i}_k} \in L^2(\Sigma_k)$ and $\Sigma_k=\{(t_1,\ldots,t_k)
  \in [0,T]^k: 0<t_1<\cdots < t_k<T\}$ is the simplex of $[0,T]^k$. 
\end{prop}

 Let  $|\mathbf{i}_k|:=\sum_{j=1}^k i_j$.
Due to the isometry property it holds
\[ \|L_k^{ \mathbf{i}_k}(f)\|^2=\kappa^{|\mathbf{i}_k|}\|f\|^2_{\Sigma_k},\] and
for any $f \in L^2(\Sigma_k),$  $g \in L^2(\Sigma_m)$,   $\mathbf{i}_k \in \{0,1\}^k$, and $\mathbf{j}_m \in \{0,1\}^m$ we have (see
\cite[Proposition 1.1]{LSUV_02})
\begin{align*}
  \e[L_k^{\mathbf{i}_k}(f)L_m^{\mathbf{j}_m}(g)]
  =\left\{
    \begin{array}{cc}
      \kappa^{|\mathbf{i}_k|}\int_{\Sigma_k}f(t_1,\cdots,t_k)g(t_1,\cdots,t_k)dt_1\cdots
        dt_k
      & \mbox{ if $ \mathbf{i}_k = \mathbf{j}_m$}\\
      0 & \mbox{ otherwise.}\\
      \end{array}\right.
  \end{align*}
  Then, $\|F\|^2=\e[F]^2+\sum_{k\ge 1}\sum_{\mathbf{i}_k } \kappa^{|\mathbf{i}_k|}\|f_{\mathbf{i}_k}\|^2_{L^2(\Sigma_k)}$. 
The chaos approximation of $F$ up to order
$p$ is defined by
\begin{align}\label{eq5}
    \cc_p(F)&:=\e(F)+\sum_{k=1}^p\sum_{\mathbf{i}_k  }   L_k^{ {\mathbf{i}_k}}
  (f_{\mathbf{i}_k})
   \end{align}
  and  $P_k(F):= \sum_{\mathbf{i}_k} L_k^{ {\mathbf{i}_k}} 
  (f_{\mathbf{i}_k})$ is the Wiener chaos of order $k$ of $F$. We have 
  \begin{align}\label{L-two-of-P-n} 
  \e[(P_k(F))^2]=\sum_{\mathbf{i}_k  } \kappa^{|\mathbf{i}_k|} \|f_{\mathbf{i}_k}\|^2_{\Sigma_k}.
  \end{align}

\begin{itemize}
\item Let $f\in L^2(\Sigma_k)$ and $j \in \{0,1\}$. Following \cite{LSUV_02},
  we define the derivative of $L_k^{\mathbf{i}_k }(f)$ w.r.t. the Brownian
  motion and the Poisson process as the element of $ L^2(\Omega \times [0,T])$ given by
  \begin{align}\label{der_L}
    D^{(j)}_t L_k^{\mathbf{i}_k}(f)&=\sum_{l=1}^k \ind_{\{i_l=j\}}
    L_{k-1}^{i_1,\cdots,
      \widehat{i_l},\cdots,i_k}(f(\underbrace{\cdots}_{l-1},t,\cdots)),
  \end{align}
  where $\widehat{i}$ means that the $i$-th index is omitted.
\item Let $j\in \{0,1\}. $  We extend the definition of   $D^{(j)}$ to 
  \begin{align*}
    Dom\;D^{(j)}:=\left \{F \in L^2(\m F_T) \mbox{ satisfying \eqref{eq0} and }\sum_{k=1}^{\infty} \sum_{\mathbf{i}_k  }    \sum_{l=1}^k
    \ind_{\{i_l=j\}} \, \kappa^{|\mathbf{i}_k |} \| f_{{\mathbf{i}_k  }}\|^2_{\Sigma_k}< \infty \right \}.
  \end{align*}

    If  $F \in Dom\;D^{(j)}$ then
  \[\|F\|^2_{Dom\;D^{(j)}} := \e |F|^2 + \kappa^j \e\int_0^T |D^{(j)}_t F|^2 dt < \infty .\]
\item  $F$ with chaotic representation \eqref{eq0} belongs to $Dom\;D  =:  \D^{1,2}$ if $F$
  belongs to \\ $Dom \;D^{(0)} \cap Dom \; D^{(1)}$, i.e.
  \begin{align*}
 \|F\|^2_{\mathbb{D}^{1,2}} :=   \e |F|^2   +   \sum_{k=1}^{\infty} k \sum_{\mathbf{i}_k}   \kappa^{|\mathbf{i}_k|} \| f_{\mathbf{i}_k }\|^2_{\Sigma_k}< \infty.
  \end{align*}
  More generally, we define $\D^{m,2}$ as follows:
\item  Let $m \ge 1$. We say that $F$ satisfying \eqref{eq0} belongs to
  $\D^{m,2}$ if it holds
  \begin{align*}
    \norm{F}^2_{\mathbb{D}^{m,2}}:=\e |F|^2 + \sum_{l=1}^m  \sum_{k=l}^{\infty} \frac{k!}{(k-l)!}
   \sum_{\mathbf{i}_k} \kappa^{|\mathbf{i}_k|} \| f_{\mathbf{i}_k}\|^2_{\Sigma_k}< \infty.
     \end{align*} 
    We recall $$\mathbb{D}^{\infty,2}=\cap_{m=1}^{\infty}
  \mathbb{D}^{m,2}.$$
   We define for  $l \in \nset^*$  with $l \le m$ the seminorm $\| \cdot \|_{D^{l}}$ on $\mathbb{D}^{m,2}$ by  
   
     \begin{align}\label{not1}
     \|F\|^2_{D^{l}}:=\sum_{\substack{\mathbf{i}_l 
    }} \kappa^{|\mathbf{i}_l   |} \e
    \left(\int_0^T \cdots \int_0^T
      \left|D^{\mathbf{i}_l }_{t_1,\cdots,t_l}F\right|^2 dt_1\cdots dt_l\right) 
 =   \sum_{k=l}^{\infty}\frac{k!}{(k-l)!}  \sum_{\mathbf{i}_k} \kappa^{|\mathbf{i}_k|} \| f_{\mathbf{i}_k}\|^2_{\Sigma_k},
\end{align}
  where $D^{\mathbf{i}_l}_{t_1,\cdots,t_l}= D^{i_1}_{t_1} \cdots D^{i_l}_{t_l}  $ represents the
  multi-index Malliavin derivative.
 \begin{rem} By using this notation we have
      $\|F\|^2_{\D^{m,2}}=\e|F|^2+\sum_{l=1}^m \|F\|^2_{D^{l}}$.
  \end{rem}
   \end{itemize}
\begin{itemize}
\item  For $m \ge 1$ and $j \in \nset^*$ we define $\mathcal{D}^{m,j}$  as the
  space of all  $F \in \D^{m,2}$ such that
   \begin{align*}
    \norm{F}^j_{m,j}:=\sum_{1\le l \le m}
  \,  \sum_{\mathbf{i}_l \in \{0,1\}^l} \,\, 
   \operatorname*{ess \,sup}_{(t_1, \cdots, t_l) \in [0,T]^l} \e[|D^{\mathbf{i}_l}_{t_1,\cdots,t_l}
    F|^j]< \infty.
   \end{align*}
(Since $(\omega, t_1,...,t_l) \mapsto   (D^{\mathbf{i}_l}_{t_1,\cdots,t_l} F)(\omega)$ is regarded as an element of $L^2(\Omega \times [0,T]^l )$  w.r.t.~the measure $\p\otimes \lambda_d$
  ($\lambda_d$ denotes the Lebesgue measure on $\rset^d$) we use the  essential supremum  w.r.t.~$\lambda_d$.)
  
  \item $\mathcal{S}^{m,j}$ denotes the space of all triples of processes $(Y,Z,U)$
  belonging to $\spp^j_T(\rset)\times
  \hpp^j_T(\rset^d)\times
  \hpp^j_T(\rset)$ and such that
  \begin{align*}
    \norm{(Y,Z,U)}^j_{m,j}:&=\sum_{1\le l \le m}
    \sum_{\substack{  \mathbf{i}_l}}
    \operatorname*{ess \,sup}_{(t_1, \cdots, t_l) \in [0,T]^l} \norm{(D^{\mathbf{i}_l}_{t_1, \cdots, t_l}
       Y,D^{\mathbf{i}_l}_{t_1, \cdots, t_l} Z,D^{\mathbf{i}_l}_{t_1, \cdots, t_l} U)}^j_{\lp^j}<\infty,
   \end{align*}
   where $\| \cdot\|^j_{\lp^j}$ has been defined in \eqref{norme_L2_YZ}. We  denote $\mathcal{S}^{m,\infty}:=\cap_{j = 1}^\infty \mathcal{S}^{m,j}$.  
\end{itemize}

\begin{rem}\label{rem70}
 If 
$F:=g(G)$, where $g:\rset
  \rightarrow \rset$ is a $\cc^1_b$ function and $G \in \D^{1,2},$ we have  (following
  \cite[Proposition 5.1]{GL_11}) that
  \begin{align*}
    (D_t^{(0)} F,D_t^{(1)} F)=(g'(G)D^{(0)}_t G,\,g(G+D_t^{(1)}G)-g(G)).
  \end{align*} 
  Moreover, using Notation \eqref{not1}, we get
  \begin{align*}
    \| F\|^2_{D^1}=\|D^{(0)} F\|^2_{L^2(\Omega \times
      [0,T])}+\kappa \|D^{(1)} F\|^2_{L^2(\Omega \times
      [0,T])} \le \|g'\|^2_{\infty} \| G\|^2_{D^1}.
  \end{align*}
  More generally, if $g:\rset
  \rightarrow \rset$ is a $\cc^m_b$ function and $G \in \D^{m,2}$, we have
  \begin{align*}
    \| F\|^2_{D^m} \le C(m, \{\|g^{(k)}\|_{\infty}\}_{k \le m}, \|G\|_{\mathbb{D}^{m,2}}),
  \end{align*} 
where $C(m, \{\|g^{(k)}\|_{\infty}\}_{k \le m}, \|G\|_{\mathbb{D}^{m,2}}\!)$ is a constant depending on $m, \{\|g^{(k)}\|_{\infty}\!\}_{k \le m}$ and $\|G\|_{\mathbb{D}^{m,2}}.$
\end{rem}

  \begin{lemme}\label{lem2}
  Let $1 \le m \le p+1$ and $F \in \D^{m,2}$. We have
  \begin{align*}
    \e[|F-\cc_p(F)|^2] \le \frac{\norm{F}^2_{D^{m}}}{ (p+2-m) \cdots(p+1) }.
  \end{align*}
\end{lemme}

\begin{proof}Using \eqref{L-two-of-P-n}, we get
  \begin{align*}
    \e[|F-\cc_p(F)|^2] &=\sum_{k \ge p+1} \e[P_k(F)^2]=\sum_{k \ge p+1} \sum_{\mathbf{i}_k}
    \kappa^{|\mathbf{i}_k|} \|f_{\mathbf{i}_k}\|^2_{\Sigma_k}\\
    &=\sum_{k \ge p+1}\frac{k!}{(k-m)!} \frac{(k-m)!}{k!} \sum_{\mathbf{i}_k}
    \kappa^{|\mathbf{i}_k|} \|f_{\mathbf{i}_k}\|^2_{\Sigma_k}\\
    &\le \frac{1}{ (p+2-m) \cdots(p+1)}\sum_{k \ge p+1} \frac{k!}{(k-m)!}\sum_{\mathbf{i}_k}
    \kappa^{|\mathbf{i}_k|} \|f_{\mathbf{i}_k}\|^2_{\Sigma_k}. \\
    &\le \frac{1}{(p+2-m) \cdots(p+1)}\sum_{k \ge m} \frac{k!}{(k-m)!}\sum_{\mathbf{i}_k}
    \kappa^{|\mathbf{i}_k|} \|f_{\mathbf{i}_k}\|^2_{\Sigma_k}. \\
     \end{align*}

    \end{proof}

\subsubsection{Multiple integrals} \label{multiple}
In the following, $\lambda$ denotes the Lebesgue measure. Setting 
\[
  M(ds,dx) := dG_0(s)  d\delta_0(x) + dG_1(s) d\delta_1(x)
\] 
we get an independent random measure in the sense of It\^o (see  \cite{Ito_56}). There exists a chaotic representation by multiple integrals w.r.t.~this random measure $M$ which is  equivalent to Proposition \ref{CRP-for-iterated}.

\begin{prop} (\cite{Ito_56}) \label{chaos-expansion}
Any $F \in L^2(\m F_T)  $ can be represented as 
\begin{align}\label{dec_F_I}
F = \e [F] + \sum_{k=1}^\infty I_k(g_k),
\end{align}
with $g_k \in (L^2)^{\otimes k}( \lambda \otimes(\delta_0 + \kappa \delta_1)):=   (L^2)^{\otimes k}([0,T]\times\{0,1\}, \m B ([0,T]) \otimes 2^{\{0,1\}}, \lambda \otimes(\delta_0 + \kappa \delta_1))$. 
This representation is unique if we assume that the functions $g_k(z_1, ...,z_k)$ with $z_i=(t_i,x_i) \in [0,T]\times\{0,1\}$ are symmetric. 
\end{prop}
\bigskip
For the definition of the  multiple integrals  $I_k(g_k)$ we refer to  \cite{Ito_56} or \cite{Sole}.  But  using  the  result that the representations  in Proposition \ref{CRP-for-iterated} and  \ref{chaos-expansion}  are both unique we conclude  for symmetric $g_k$  the relation
\begin{align} \label{multiple-iterated-rel}
 I_k(g_k)  = k! \sum_{\mathbf{i}_k}L_k^{\mathbf{i}_k}(g_k((\cdot, i_1),\cdots,(\cdot, i_k))),
 \end{align} 
  where $\mathbf{i}_k$ is defined in Proposition \ref{CRP-for-iterated}  and $$g_k(((t_1, i_1),\cdots,(t_k, i_k))) = k! f_{\mathbf{i}_k}(t_1,\ldots,t_k)   \quad \text{ on } 
\Sigma_k$$  with $f_{\mathbf{i}_k}$ from Proposition \ref{CRP-for-iterated}.

Moreover,   for symmetric $g_k \in (L^2)^{\otimes k}( \lambda \otimes(\delta_0 + \kappa \delta_1))$ and $f_m \in (L^2)^{\otimes m}( \lambda \otimes(\delta_0 + \kappa \delta_1))$
the relation 
  \begin{align} \label{isometry-formula}
  \e[ I_k(g_k)   I_m(f_m)   ]    =\left\{
    \begin{array}{cc}
    k!  \langle g_k,f_k \rangle_{ (L^2)^{\otimes k}( \lambda \otimes(\delta_0 + \kappa \delta_1))}       & \mbox{ if $ k = m$}\\
      0 & \mbox{ otherwise,}\\
      \end{array}\right.
  \end{align}
holds true.
If $F \in \D^{m,2}$, we combine \eqref{dec_F_I},
 \eqref{multiple-iterated-rel} and \cite[Definition 1.7]{LSUV_02} (which
 extends \eqref{der_L} to functions defined on $L^2([0,T])^k$) to get
 \begin{align} \label{kernel-Malliavin-rel}
 g_k((t_1,i_1),\ldots,(t_k,i_k))= \frac{1}{k!}  \e D^{\mathbf{i}_k}_{t_1,\ldots,t_k}F, \quad  k \le m.
  \end{align}
On the other hand, this can be easily derived  if one takes into account that for $F= \e [F] + \sum_{k=1}^\infty I_k(g_k)$ we have    $  D^{\mathbf{i}_1}_{t_1} F=  \sum_{k=1}^\infty k I_{k-1}(g_k((t_1,i_1), \cdot)),$ and  that the expectation of any multiple integral of order $k>0$ is zero while $I_0$ is the identity map.  
\bigskip

For the implementation of the numerical scheme  we will use Hermite and Charlier polynomials.  In order to do so, we provide a 
chaotic representation consisting only of iterated integrals of the form   $L^{0,...,0}$ and  $L^{1,...,1}$ for which the relations \eqref{iterated-Hermite}
and \eqref{iterated-Charlier} below can be used. \bigskip

Use $\{p_0, p_1\}= \{ \ind_{\{0\}},  \frac{1}{\sqrt{\kappa }} \ind_{\{1\}}  \}$ as orthonormal basis of $L^2(\{0,1\}, 2^{\{0,1\}}, \delta_0 + \kappa \delta_1)$ and 
fix an orthonormal basis $\{e_k\}_{k \in \nset} $ for $L^2([0,T], \m B ([0,T]),  \lambda)$. By setting
\[
   e[(k_1,i_1),\ldots,(k_m,i_m)]:= (e_{k_1}\otimes p_{i_1}) \otimes \ldots \otimes (e_{k_m}\otimes p_{i_m}), \quad k_j \in \nset, i_j \in \{0,1\}
\] 
we get an orthonormal basis of  $(L^2)^{\otimes m}( \lambda \otimes(\delta_0 + \kappa \delta_1)).$  The symmetrizations
\begin{equation} \label{symmetric-basis}
   \tilde e[(k_1,i_1),\ldots,(k_m,i_m)]:= \frac{1}{m!} \sum_{\pi \in \m S_m} e[(k_{\pi(1)},i_{\pi(1)}),\ldots,(k_{\pi(m)},i_{\pi(m)})],\quad k_j \in \nset, i_j \in \{0,1\} 
\end{equation}
form an {\it  orthogonal} basis of $\tilde{ (L^2)}^{\otimes m}( \lambda \otimes(\delta_0 + \kappa \delta_1)),$  the subspace of symmetric functions from
$(L^2)^{\otimes m}( \lambda \otimes(\delta_0 + \kappa \delta_1)).$ 

We also will use the notation 
\[
   \tilde e[k_1,\ldots,k_m]:=  \frac{1}{m!} \sum_{\pi \in \m S_m} e_{k_{\pi(1)}} \otimes \cdots \otimes   e_{k_{\pi(m)}}, \quad k_j \in \nset,
\]
where $\m S_m$ stands for the set of all permutations of $\{1,...,m\}.$
\begin{prop}\label{chaos-decomposition}  
Any $F \in L^2(\m F_T)  $ can be represented as 
\begin{align*}
F& = \e [F] + \sum_{k=1}^\infty \sum_{l=0}^k \, \sum_{{\bf k}_l \in \nset^l  }  \,\,  \sum_{{\bf j}_{k-l} \in \nset^{k-l}}   
  d_{{\bf k}_l,   {\bf j}_{k-l}} L^{0,...,0}_l(\tilde e[k_1,\ldots,k_l]) \, L^{1,...,1}_{k-l}(\tilde e[j_1,\ldots,j_{k-l}]),
\end{align*} 

where $d_{{\bf k}_l,   {\bf j}_{k-l}}= \frac{ l!(k-l)! \langle g_k,
  e[(k_1,0),....(k_l,0)]\otimes
  e[(j_1,1),...,(j_{k-l},1)]\rangle_{(L^2)^{\otimes k}}  }{ \kappa^{\frac{k-l}{2}} \|  \tilde e[(k_1,0),....(k_l,0),(j_1,1),...,(j_{k-l},1) ] \|^2_{(L^2)^{\otimes k}( \lambda \otimes(\delta_0 + \kappa \delta_1))}}.$

\end{prop}

\begin{proof}

According to  \cite[ Theorem 1]{Ito_56}  a permutation of the coordinates of the kernels  does not change the multiple integral, i.e.   for any $\pi \in \m S_k$ we have 
\[ 
I_k( \tilde e[(k_1,i_1),\ldots,(k_k,i_k)]) = I_k(e[(k_{\pi(1)},i_{\pi(1)}),\ldots,(k_{\pi(k)},i_{\pi(k)})]).
 \]
 For any  $\pi$ with $( i_{\pi(1)},\ldots,i_{\pi(k)})=(0,\ldots,0,1,\ldots,1)$ (we assume that $(i_1, \ldots,i_k)$ contains $l$ zeros)
 it holds by   the product formula 
for multiple integrals  (see  \ref{LeeShihformula} or \cite[Theorem 3.6]{LS_04}) 
 \begin{equation} \label{product-formula-cons}
I_k( \tilde e[(k_1,i_1),\ldots,(k_k,i_k)]) = I_l(e[(k_{\pi(1)},0),\ldots,(k_{\pi(l)},0)]) I_{k-l}(e[(k_{\pi(l+1)},1),\ldots,(k_{\pi(k)},1)])
 \end{equation} 
since 
\[e[(k_{\pi(1)},i_{\pi(1)}),\ldots,(k_{\pi(k)},i_{\pi(k)})] = e[(k_{\pi(1)},0),\ldots,(k_{\pi(l)},0)] \otimes e[(k_{\pi(l+1)},1),\ldots,(k_{\pi(k)},1)],\]
and  for the contraction-identification $ \otimes_m^r$ (for the definition see \eqref{contraction-identification}) it holds
\[e[(k_{\pi(1)},0),\ldots,(k_{\pi(l)},0)] \otimes_m^r e[(k_{\pi(l+1)},1),\ldots,(k_{\pi(k)},1)] =0  \]
if $r \neq 0$ or $m \neq 0.$ Since
\[
e[(k_{\pi(l+1)},1),\ldots,(k_{\pi(k)},1)]= \frac{1}{\kappa^{\frac{k-l}{2}}} e[k_{\pi(l+1)},\ldots,k_{\pi(k)}],
\]
we conclude from \eqref{product-formula-cons} and   \eqref{multiple-iterated-rel}   that
 \begin{align} \label{multiple=product-of-iterated}
I_k( \tilde e[(k_1,i_1),\ldots,(k_k,i_k)]) =\frac{l!(k-l)!}{\kappa^{\frac{k-l}{2}}} L^{0,...,0}_l(\tilde e[k_{\pi(1)},\ldots,k_{\pi(l)}]) L^{1,...,1}_{k-l}(\tilde e[k_{\pi(l+1)},\ldots,k_{\pi(k)}]).
\end{align}

The symmetric functions $g_k$ from Proposition \ref{chaos-expansion} can be written as
\begin{align} \label{g_k-expansion}
     g_k =& \sum_{l=0}^k  \sum_{{\bf k}_l} \sum_{{\bf j}_{k-l}} \langle g_k, e[(k_1,0),...,(k_l,0)]\otimes e[(j_1,1),...,(j_{k-l},1)]\rangle_{(L^2)^{\otimes k}}  \nonumber \\
     &\quad\quad \quad\quad\quad \quad \times \tilde e[(k_1,0),...,(k_l,0),(j_1,1),...,(j_{k-l},1)] c_{{\bf k}_l,   {\bf j}_{k-l}},
\end{align}
where we sum over all  ${\bf k}_l \in \nset^l$ and ${\bf j}_{k-l} \in \nset^{k-l}$ and  
\[c_{{\bf k}_l,   {\bf j}_{k-l}}  = \|  \tilde
e[(k_1,0),...,(k_l,0),(j_1,1),...,(j_{k-l},1) ] \|^{-2}_{(L^2)^{\otimes k}(
  \lambda \otimes(\delta_0 + \kappa \delta_1))} \]
denotes the normalizing factor. 
Abbreviating 
\[
    d_{{\bf k}_l,   {\bf j}_{k-l}} := \frac{l!(k-l)!}{\kappa^{\frac{k-l}{2}}} \langle g_k, e[(k_1,0),...,(k_l,0)]\otimes e[(j_1,1),...,(j_{k-l},1)]\rangle_{(L^2)^{\otimes k}} \,  c_{{\bf k}_l,{\bf j}_{k-l}} 
\]
 we conclude from Proposition \ref{chaos-expansion}, \eqref{multiple=product-of-iterated}  and  \eqref{g_k-expansion}  the orthogonal decomposition
 \begin{align}\label{dec_F_mult_int}
F& = \e [F] + \sum_{k=1}^\infty \sum_{l=0}^k  \sum_{{\bf k}_l} \sum_{{\bf j}_{k-l}}     d_{{\bf k}_l,   {\bf j}_{k-l}} L^{0,...,0}_l(\tilde e[k_1,\ldots,k_l]) \, L^{1,...,1}_{k-l}(\tilde e[j_1,\ldots,j_{k-l}]).
\end{align} 
 
 \end{proof}
\bigskip

\begin{lemme}\label{c_kj-norm}
Fix $N \in \nset^*$ and  let
\[e[(k_1,0),...,(k_l,0),(j_1,1),...,(j_{k-l},1) ] =  \bigotimes_{i=1}^N
(e_i\otimes p_0)^{\otimes n^B_i}\otimes  \bigotimes_{j=1}^N (e_j\otimes p_1)^{\otimes n^P_j}, \]
i.e.   $n^B_i$ and $n^P_i$  $(1\le i \le N)$ denote the multiplicities  of the functions $e_i\otimes p_0$ and $e_i\otimes p_1$, respectively, so that
$|{\bf n}^B|=|(n_1^B,...,n_N^B)|=l$ and $|{\bf n}^P|=|(n_1^P,...,n_N^P)|=k-l.$ Let ${\bf n}^A!:=n^A_1!\cdots n^A_N!$ for $A=B,P$ and define ${\bf n}:=({\bf n}^B,{\bf n}^P)$ so that $|{\bf n}|=|{\bf n}^B|+|{\bf n}^P|.$
Then 

 \[c_{{\bf k}_l,   {\bf j}_{k-l}}  = \|  \tilde
e[(k_1,0),...,(k_l,0),(j_1,1),...,(j_{k-l},1)] \|^{-2}_{(L^2)^{\otimes k}(
  \lambda \otimes(\delta_0 + \kappa \delta_1))}  = \frac{|{\bf n}|!}{{\bf n}^B ! \, {\bf n}^P!}.\]
\end{lemme}

\begin{proof}
 To compute $c_{{\bf k}_l,   {\bf j}_{k-l}} $ notice that the functions $h_j:= (e_{k_j}\otimes p_{i_j})$ and $h_{j'} $  $(1 \le j, j' \le k)$ are either 
equal or orthogonal in $L^2( \lambda \otimes(\delta_0 + \kappa \delta_1)).$ Denoting  
\[ e[(k_1,0),...,(k_l,0),(j_1,1),...,(j_{k-l},1) ] = h_1\otimes \cdots \otimes h_k \] 
 yields
\begin{align*}
& \|  \tilde e[(k_1,0),...,(k_l,0),(j_1,1),...,(j_{k-l},1)] 
 \|^2_{(L^2)^{\otimes k}(
  \lambda \otimes(\delta_0 + \kappa \delta_1))} \\
   &= \bigg \|  \frac{1}{k!} \sum_{\pi \in \m S_k} h_{\pi(1)} \otimes \cdots \otimes   h_{\pi(k)} \bigg \|^2_{(L^2)^{\otimes k}(
  \lambda \otimes(\delta_0 + \kappa \delta_1))} \\
  &= \frac{n^B ! n^P! }{k!} \bigg \|  \bigotimes_{i=1}^N (e_i\otimes p_0)^{\otimes n^B_i}\otimes  \bigotimes_{j=1}^N (e_j\otimes p_0)^{\otimes n^P_j}  \bigg \|^2_{(L^2)^{\otimes k}(
  \lambda \otimes(\delta_0 + \kappa \delta_1))} \\
  &=  \frac{{\bf n}^B ! {\bf n}^P! }{k!}.
\end{align*}
\end{proof}
\bigskip

 \begin{rem}\label{rem0}
   We deduce from \eqref{dec_F_mult_int} that
   \begin{align*}
     \cc_p(F)& = \e [F] + \sum_{k=1}^p \sum_{l=0}^k  \sum_{{\bf k}_l} \sum_{{\bf j}_{k-l}}     d_{{\bf k}_l,   {\bf j}_{k-l}} L^{0,...,0}_l(\tilde e[k_1,\ldots,k_l]) 
             L^{1,...,1}_{k-l}(\tilde e[j_1,\ldots,j_{k-l}]).
   \end{align*}
 \end{rem}
 
 In order to compute the expectation of products of multiple integrals (see formula \eqref{expectation-product-formula} below)
 we introduce some notation following \cite{LPST_14}, \cite{S_13}, \cite{peccati}
 and \cite{nualart_06}.
 \begin{itemize}
 \item If $n \in \nset^*$   then $[n]:=\{1,\ldots, n\}.$
 \item For $J \subseteq [n]$ we denote by $O^J_n$ the singleton containing that $x \in \{0,1\}^n$ for which
$x_i =0 \iff i \in J$ holds.
\item If $n_1, \ldots,n_l$ ($l\in\nset^*$) are given and $n:=n_1+\cdots+n_l$ we will
denote by $\Psi$ the `natural' partition of $[n]$  given by the summands $n_i:$
\begin{align*}
\Psi &:=\{ \Psi_1, \ldots, \Psi_l\} \\
&:=\{ \{1,\ldots, n_1\}, \ldots, \{n_1+\cdots+n_{l-1}+1, \ldots, n\}\}.
\end{align*} 
\item Let $\Pi_n$ denote the set of all partitions of $[n]$ (a partition means here a  set of disjoint non-empty subsets of $[n]$ such that their union is $[n]$) and $\Pi_n^*$ denote the set of all subpartitions of $[n]$  (any  set of disjoint non-empty subsets of $[n]$  is a subpartition).
\item  Let  $\Pi(n_1,\ldots,n_l) \subseteq \Pi_n$ (respectively $\Pi^*(n_1,\ldots,n_l) \subseteq \Pi^*_n$) denote the set of all $\sigma \in \Pi_n$ (respectively   $\sigma \in \Pi^*_n$)  with $|\Psi_i \cap J| \le 1$   for  $1\le i \le l$ and all $J \in \sigma.$  
\item Let $\Pi_{\ge 2}(n_1,\ldots,n_l) $    (respectively $\Pi_{=2}(n_1,\ldots,n_l)$) denote the set of all $\sigma \in \Pi(n_1,\ldots,n_l)$  with $|J| \ge  2$ (respectively $|J| = 2$) for all $J \in \sigma.$ 
\item In order to distinguish between integration w.r.t. the Brownian motion and  compensated Poisson process we consider for $J^B \subseteq [n]$ ($J^B$ will stand for integration w.r.t~the Brownian motion)
and introduce $\Pi_{=2,\ge 2}(J^B; n_1,\ldots,n_l) $  as the set of all pairs ($\tau, \sigma$) of subpartitions from $\Pi^*_n(n_1,\ldots,n_l)$ such that 
 for all $J \in \tau$:  $|J| = 2$ and $\bigcup_{J \in \tau}J=J^B$  as well as  for all $J \in \sigma$: $|J| \ge 2$  and  $\bigcup_{J \in \sigma}J=[n] \setminus J^B.$
 
 \item For $\tau \in \Pi_n^*$ let  $|\tau| =\#\{J\subseteq [n] : J \in \tau\}$ i.e. the number of its blocks
   and  $\|\tau\| := \#\bigcup_{J \in \tau}J.$

 \item For $(\tau, \sigma) \in \Pi_{=2,\ge 2}(J^B; n_1,\ldots,n_l)$ and $f: ([0,T] \times \{0,1\})^n \to \rset$  we define 
       $f_{\tau \cup\sigma} : [0,T]^{{|\tau|+|\sigma|}} \to \rset$ by identifying the time variables of each block of $\tau \cup\sigma$
       and setting $x_i=0$ for $i \in \bigcup_{J \in \tau}J  $ and $x_i=1 $ for $i \in \bigcup_{J \in \sigma}J.  $ In order to make this map unique we identify first 
        the time variables of that block of $\tau \cup\sigma$ which contains the smallest number and denote all identified variables by $t_1.$ Next we choose  from the remaining  blocks  that  one containing the smallest number and use $t_2$ for all identified variables and so on. 
 \end{itemize}

 Example:  Let $n_1= 2, n_2=2$ and $n_3=3.$ Then $\Psi= \{\{1,2\},\{3,4\},\{5,6,7\}\}.$
 If  $J^B=\{2,4,6,7\}$ and  $\tau =\{\{2,6\},\! \{4,7\}\!\}$, $\sigma=\! \{\{1,3,5\}\}$
 we change by $\tau \cup\sigma$  the function $f((t_1,x_1), \cdots,$ $ (t_7
 ,x_7))$  
 into 
 \[
 f_{\tau \cup\sigma}(t_1,t_2,t_3 )= f((t_1,1),(t_2,0),( t_1,1),( t_3,0),(t_1,1),(t_2,0),(t_3,0)). 
 \]
 \begin{prop} \label{E-of-chaos-products} Let $f_{n_i} \in (L^2)^{\otimes {n_i}}( \lambda \otimes(\delta_0 + \kappa \delta_1))$ 
($n_i \in \nset$ for $1 \le i \le l$) be symmetric and assume that for all $ (\tau,\sigma) \in \Pi_{=2,\ge 2}(J^B; n_1,\ldots,n_l)$ it holds
 \[
    \int_{[0,T]^{|\tau|+|\sigma|}} \bigg ( \bigotimes_{i=1}^l | f_{n_i} |\bigg )_{\tau \cup\sigma}  d\lambda^{|\tau|+|\sigma|} <\infty.
 \]
Then 
\begin{align}  \label{expectation-product-formula}
\e \Pi_{i=1}^l    I_{n_i}(f_{n_i})  =  \sum_{J^B \in [n]} \,\, \sum_{(\tau,\sigma) \in \Pi_{=2,\ge 2}(J^B; n_1,\ldots,n_l)}
 \kappa^{|\sigma|}  \int_{[0,T]^{|\tau|+|\sigma|}} \bigg ( \bigotimes_{i=1}^l f_{n_i} \bigg )_{\tau \cup\sigma}  d\lambda^{|\tau|+|\sigma|} .
\end{align} 
\end{prop}

\begin{proof}
Let us assume for the moment that the $f_{n_i}$ are of the form  
 \begin{align} \label{tensor-f}
f_{n_i}((t_1,x_1), \cdots,(t_{n_i},x_{n_i}))   = \Pi_{k=1}^{n_i}  d_i(t_k,x_k)
\end{align} 
for some $d_i \in  L^2( \lambda \otimes(\delta_0 + \kappa \delta_1)).$

 If  $J_i \subseteq [n_i] $ and  $n_i^0= \#J_i,$ we let   $ I^B_{n_i^0}$ denote the multiple integral of order $n_i^0$  w.r.t.~the Brownian motion and 
 $ I^P_{n_i^1}$ the multiple integral of order $n_i^1$ ($n_i^1:= n_i -n_i^0$) w.r.t.~the compound Poisson process. Similar to 
\eqref{product-formula-cons} we get

\[
   I_{n_i} (   d_i^{ \otimes n_i}   \ind_{O^{J_i}_{n_i}}) =  I^B_{n_i^0}( [d_i(\cdot,0)]^{\otimes  n_i^0}) I^P_{n_i^1}( [d_i(\cdot,1)]^{  \otimes  n_i^1 } ). 
 \]
 Consequently, since

 \[
   \sum_{J_i \subseteq [n_i] } \ind_{O^{J_i}_{n_i}}(x) = 1, \quad x\in \{0,1\}^{n_i},
\]
 
 \begin{align*}
\e \Pi_{i=1}^l    I_{n_i}(f_{n_i})  & =   \sum_{J_1 \subseteq [n_1], \ldots, J_l \subseteq [n_l]}   \e \Pi_{i=1}^l    I_{n_i}(   (\otimes_{k=1}^{n_i}  d_i) \ind_{O^{J_i}_{n_i}}  ) \\
& =  \sum_{J_1 \subseteq [n_1], \ldots, J_l \subseteq [n_l]}   \e \Pi_{i=1}^l \big \{ I^B_{n_i^0}( [d_i(\cdot,0)]^{\otimes  n_i^0})
 I^P_{n_i^1}( [d_i(\cdot,1)]^{  \otimes  n_i^1 } ) \big \} \\
& =  \sum_{J_1 \subseteq [n_1], \ldots, J_l \subseteq [n_l]}   \e [ \Pi_{i=1}^l  I^B_{n_i^0}( [d_i(\cdot,0)]^{\otimes  n_i^0})]  \,\,  \e [ \Pi_{i=1}^l  I^P_{n_i^1}( [d_i(\cdot,1)]^{  \otimes  n_i^1 } )  ].
\end{align*}

From \cite[Corollary 7.3.2]{peccati} we conclude
 \begin{align*}
\e  \Pi_{i=1}^l I^B_{n_i^0}( [d_i(\cdot,0)]^{\otimes  n_i^0}) =  \sum_{\tau \in \Pi_{=2}(n^0_1,\ldots,n^0_l) }   \int_{[0,T]^{|\tau|}} \bigg ( \bigotimes_{i=1}^l     
 d_i^{\otimes  n_i^0} \bigg )_{\tau}  d\lambda^{|\tau|},
\end{align*} 
while \cite[Theorem 3.1]{LPST_14} (see also \cite[Section 3.2]{S_13}) implies 
 \begin{align*}
 \e  \Pi_{i=1}^l  I^P_{n_i^1}( [d_i(\cdot,1)]^{  \otimes  n_i^1 } )   =  \sum_{\sigma \in \Pi_{\ge 2}(n^1_1,\ldots,n^1_l) }  
    \kappa^{|\sigma|} \int_{[0,T]^{|\sigma|}} \bigg ( \bigotimes_{i=1}^l     
 d_i^{\otimes  n_i^1} \bigg )_{\sigma}  d\lambda^{|\sigma|}.
\end{align*} 

So we have shown relation \eqref{expectation-product-formula} for the special situation \eqref{tensor-f} where each $f_{n_i}$ is given as tensor product
$d_i^{\otimes n_i}. $ The general assertion follows by approximation using the multilinear nature  of    \eqref{expectation-product-formula} w.r.t. ($f_{n_1},\ldots ,f_{n_l}$).
\end{proof}

\subsection{Hermite and Charlier polynomials}
\subsubsection{Hermite polynomials}\label{sec:hermite_pol}

Let us introduce the Hermite polynomials $(K_m)_{m \in \nset}$ defined by
\begin{align*}
  e^{xt-\frac{t^2}{2}}=\sum_{m \ge 0} K_m(x) t^m, \quad t,x \in \rset.
\end{align*}
With the convention $K_{-1} = 0$  we have the relations $K''_m(x)-x K'_m(x)+mK_m(x)=0$ 
and $K'_m(x)=K_{m-1}(x)$, for all $m  \in \nset$. The normalized sequence $(\sqrt{m!} K_m)_{m \in \nset}$ forms an orthonormal basis 
in $\lp^2(\rset,\mu)$, where $\mu$ denotes the  normalized centered Gaussian
 measure.

 Every square integrable random variable $F$, measurable with respect to ${\m
   F}^B_T$, admits the following orthogonal decomposition
\begin{equation}
  \label{eq4}
  F = d_0 + \sum_{k=1}^\infty \sum_{|{\bf n}| = k} d_k^{\bf n} \, \prod_{i\geq 1} K_{n_i}\left(\int_0^T e_i(s)dB_s\right),
\end{equation}
where ${\bf n}=\{n_i\}_{i\geq 1}$ is a sequence of non-negative integers, $|{\bf n}|:=\sum_{i\geq 1} n_i$  and $\{e_i\}_{i \ge 1}$ is an orthonormal basis of $L^2(0,T).$  
Taking into account the normalization of the Hermite polynomials we use, we get
\begin{equation*}
	d_0 = \e\left[F\right],\qquad d_k^{\bf n} = {\bf n} !\, \e\left[F \times \prod\nl_{i\geq 1} K_{n_i}\left(\int_0^T e_i(s)dB_s\right)\right],
\end{equation*}
where ${\bf n} ! = \prod_{i\geq 1} (n_i !)$. \bigskip

Now we choose $N \in \nset$ and let $\{\ov{t}_0,\ov{t}_1,\cdots,\ov{t}_N\}$ be a regular grid of $[0,T]$, i.e. $\forall i
 \in \{0,\ldots,N\},$ $\ov{t}_i=ih$ where $h=\frac{T}{N}.$ From now on we will use a fixed orthonormal basis $\{e_i\}_{i \ge 1}$ of $L^2(0,T):$ we set
\begin{equation} \label{basis-fixed}
 e_i(t):=\frac{1}{\sqrt{h}}\ind_{]\ov{t}_{i-1},\ov{t}_i]}(t), \quad i \in \{0,\ldots,N\}
\end{equation}
 and complete  this sequence to a basis in $L^2(0,T),$ for example, by using the Haar basis on each interval $]\ov{t}_{i-1},\ov{t}_i].$ 
Let ${\bf n}^B=(n^B_1,\ldots, n^B_N)$ be the
 vector of non-negative integers such that $|{\bf n}^B|=k$.
Then (see \cite[Proposition 5.1.3]{privault})
  \begin{align}\label{iterated-Hermite}
    L^{0,\cdots,0}_k(e_1^{\otimes n^B_1}\circ \cdots \circ
   e_N^{\otimes n^B_N})=  \frac{{\bf n}^B! }{|{\bf n}^B|!} \prod_{i=1}^N K_{n^B_i}\left(\frac{\Delta
     B_i}{\sqrt{h}}\right),
\end{align}
where $\Delta B_i=B_{\ov{t}_i}-B_{\ov{t}_{i-1}}$ and $\circ$ stands for the symmetric tensor product.

 \subsubsection{Charlier polynomials}
 \begin{definition}
   The Charlier polynomial of order $m \in \nset$ and of parameter $t\ge 0$
   is defined by
   \begin{align*}
     C_0(x,t)=1,\;C_1(x,t)=x-t, \quad x \in \rset
   \end{align*}
   and by the relation
   \begin{align*}
     C_{m+1}(x,t)=(x-m-t)C_m(x,t)-mtC_{m-1}(x,t).
   \end{align*}
 \end{definition}
The sequence
$\left \{\frac{1}{\sqrt{m! (\kappa t)^m}}C_m(\cdot,\kappa t)\right \}_{m \in \nset}$ is an
orthonormal basis for $L^2(\nset,\nu_{\kappa t})$, where $\nu_{\kappa t}$
denotes the law of a Poisson random variable with parameter $\kappa t$. Let ${\bf n}^P=(n^P_1,\ldots, n^P_N)$ be the
 vector of non-negative integers such that $|{\bf n}^P|=k$.
 Using the same grid  and the same functions $\{e_i\}_{1 \le i \le N}$ as for \eqref{iterated-Hermite}, we have (see \cite[Proposition 6.2.9]{privault})
\begin{align} \label{iterated-Charlier}
   L^{1,\ldots,1}_{k}(e_1^{\otimes n^P_1}\circ \cdots \circ
   e_N^{\otimes n^P_N})=
   \frac{1}{ |{\bf n}^P|!  \, h^{\frac{|{\bf n}^P|}{2}}}
   \prod_{i=1}^N C_{n^P_i}\left(\Delta
     N_i,\kappa h\right),
\end{align}
where   $\Delta N_i=N_{\ov{t}_i}-N_{\ov{t}_{i-1}}$. The following Lemma gives some useful properties of the chaos decomposition.
\begin{lemme}\label{lem4}\hfill
  \begin{itemize}
  \item  Let $F$ be a r.v. in $\lp^2(\m F_T)$. $\forall p \ge 1$, we have $
    \e(|\cc_p(F)|^2) \le \e(|F|^2)$. 
  \item Let $H$ be in $\hpp^2_T(\rset)$. We deduce from Remark \ref{rem0} that $\cc_p \left(\int_0^T H_s
      ds\right)=\int_0^T \cc_p (H_s) ds$. 
  \item For all $F \in \mathbb{D}^{1,2}$, for all $i \in \{0,1\}$ and for all $t \le r$, $D^{(i)}_t \e_r[\cc_p(F)]=\e_r[\cc_{p-1}(D^{(i)}_t F)]$,
  where $\e_r$ stands for the conditional expectation with respect to $\m F_r.$
  \end{itemize}
\end{lemme}
\subsection{Truncation of the basis}
Instead of summing over all $\mathbf{k}_l \in \mathbb{N}^l$ and
$\mathbf{j}_{ k-l} \in \mathbb{N}^{k-l}$, we only consider the $N$ first functions $\{e_1,\ldots,e_N\}$ of the basis $\{e_i\}_i$ defined in \eqref{basis-fixed}.
 This gives (together with the orthogonal projection onto the chaos up to order $p$) the following approximation of $F$
  \begin{align}\label{eq51}
     \cc^N_p(F) &= \e [F] \nonumber\\& \,\,\,+ \sum_{k=1}^p \sum_{l=0}^k  \sum_{{\bf k}_l \in
       \{1,\cdots,N\}^l} \sum_{{\bf j}_{k-l} \in \{1,\cdots,N\}^{k-l}}
                                       d_{{\bf k}_l,   {\bf j}_{k-l}}
                                       L^{0,...,0}_l(\tilde e[k_1,\ldots,k_l])
                                       L^{1,...,1}_{k-l}(\tilde                                        e[j_1,\ldots,j_{k-l}])\notag\\
    &:=\e [F]+\sum_{k=1}^p P^N_k(F).
  \end{align}
 Let us now rewrite $\cc^N_p(F)$ $(p \le N)$  in terms of Hermite and
   Charlier polynomials. From \eqref{multiple=product-of-iterated}, \eqref{isometry-formula}, \eqref{iterated-Hermite} and \eqref{iterated-Charlier} we derive using the notation of Lemma \ref{c_kj-norm} that
    \begin{align*}
  &\langle g_k, e[(k_1,0),...,(k_l,0)]\otimes e[(j_1,1),...,(j_{k-l},1)]\rangle_{(L^2)^{\otimes k}}
  \\&= \frac{{\bf n}^B!}{|{\bf n}|!  (\kappa h)^{|{\bf n}^P|/2}} \e\left(F\prod_{i=1}^N
  K_{n^B_i}\left(G_i\right)C_{{\bf n}^P_i}(Q_i,\kappa h)\right),
  \end{align*}  
   where we used $G_i:=\frac{\Delta B_i}{\sqrt{h}}$ and $Q_i:=\Delta N_i.$ 
  From  Lemma \ref{c_kj-norm} we get then 
\begin{align}\label{eq:chaos_dec}
  \cc_p^N(F)=d_0+\sum_{k=1}^p \sum_{|{\bf n}|=k} d^{\bf n}_k \prod_{i=1}^N
  K_{{\bf n}^B_i}\left(G_i\right)C_{{\bf n}^P_i}(Q_i,\kappa h),
\end{align}
where $d_0=\e(F)$ and
\begin{align}\label{eq:coef_chaos_dec}
  d^n_k:=\frac{{\bf n}^B!}{{\bf n}^P! (\kappa h)^{|{\bf n}^P|}}\e\left(F\prod_{i=1}^N
  K_{n^B_i}\left(G_i\right)C_{n^P_i}(Q_i,\kappa h)\right).
\end{align}

\begin{prop}\label{prop2} Let $F$ be a real random  variable in $\lp^2(\m F_T)$ and let $r$ be an integer in $\{1,\cdots, N\}$.
 For all $\ov{t}_{r-1}<t\leq \ov{t}_r$, we have
\begin{align*}
  &\e_t\left(\cc_p^N (F) \right)  = d_0 + \\
  &\sum_{k=1}^p \sum_{|{\bf n}(r)|=k} d_k^{\bf n}
   \left(\frac{t-\ov{t}_{r-1}}{h}\right)^{\frac{n^B_r}{2}}
   K_{n^B_r}\left(\frac{B_t-B_{\ov{t}_{r-1}}}{\sqrt{t-\ov{t}_{r-1}}}\right)C_{n^P_r}(N_t-N_{\ov{t}_{r-1}},\kappa(t-\ov{t}_{r-1})) \\& \quad\quad\quad  \times \underbrace{\left(\prod_{i<r}
       K_{n^B_i}(G_i)C_{n^P_i}(Q_i,\kappa h)\right)}_{:=A_r}, 
 \end{align*}
 \begin{align*}
 &D^{(0)}_t \e_t\left(\cc_p^N(F)\right)\\
&= h^{-1/2} \sum_{k=1}^p \sum_{\substack{|{\bf n}(r)|=k \\ n^B_r>0}} d_k^{\bf n}
   \left(\frac{t-\ov{t}_{r-1}}{h}\right)^{\frac{n^B_r-1}{2}}
   K_{n^B_r-1}\left(\frac{B_t-B_{\ov{t}_{r-1}}}{\sqrt{t-\ov{t}_{r-1}}}\right)C_{n^P_r}(N_t-N_{\ov{t}_{r-1}},\kappa(t-\ov{t}_{r-1}))A_r,
 \end{align*}
 \begin{align*}
   &D^{(1)}_t \e_t\left(\cc_p^N(F)\right)\\
   &= \sum_{k=1}^p \sum_{\substack{|{\bf n}(r)|=k \\ n^P_r>0}} d_k^{\bf n}
   \left(\frac{t-\ov{t}_{r-1}}{h}\right)^{\frac{n^B_r}{2}}
   K_{n^B_r}\left(\frac{B_t-B_{\ov{t}_{r-1}}}{\sqrt{t-\ov{t}_{r-1}}}\right)n^P_r C_{n^P_r-1}(N_t-N_{\ov{t}_{r-1}},\kappa(t-\ov{t}_{r-1}))A_r,
\end{align*}
where for $r\leq N$  ${\bf n}(r)=({\bf n}^B(r),{\bf n}^P(r))$, and ${\bf n}^A(r)$ stands for
$(n^A_1,\ldots, n^A_r)$, where $A=B$ or $P$ and ${\bf n}_r=({\bf n}_r^B,{\bf n}_r^P)$.
\end{prop}

\begin{proof}
  The first result comes from \cite[Proposition 2.7]{BL_14} for the Brownian
  part and from the fact that $\e_t(C_n(Q_r,\kappa
  h))=\e_t[I_n(\ind^{\otimes
    n}_{]\ov{t}_{r-1},\ov{t}_r]})]=C_n(N_t-N_{\ov{t}_{r-1}},\kappa(t-\ov{t}_{r-1}))$ (see
  \cite[Proposition 6.2.9]{privault}). The second
  result comes from \cite[Proposition 2.7]{BL_14}. To get the last one, we
  write $D_t^{(1)} C_{n^P_r}(N_t-N_{\ov{t}_{r-1}},t-\ov{t}_{r-1})=D_t^{(1)}
  I_{n^P_r}(\ind^{\otimes n^P_r}_{]\ov{t}_{r-1},t]})=n^P_r I_{n^P_r-1}(\ind^{\otimes
    n^P_r-1}_{[\ov{t}_{r-1},t]})$ (see \cite[Definition 6.4.1]{privault}).
  
\end{proof}

\begin{rem}\label{rem1}
  For $t={t}_r$ and $r\ge 1$, Proposition \ref{prop2} leads to
  \begin{align*}
	\e_{{t}_r}\left(\cc_p^N (F)\right)  & = d_0 + \sum_{k=1}^p \sum_{|{\bf n}(r)|=k} d_k^{\bf n} \prod_{i\leq r} K_{n^B_i}\left(G_i\right)C_{n^P_i}(Q_i,\kappa h), \\
	D^{(0)}_{{t}_r}\e_{{t}_r}\left(\cc_p^N (F)\right) & = h^{-1/2} \sum_{k=1}^p
    \sum_{\substack{|{\bf n}(r)|=k \\ n^B_r>0}} d_k^{\bf n}
    K_{n^B_r-1}\left(G_r\right)C_{n^P_r}(Q_r,\kappa h) \left(\prod_{i<r}
      K_{n^B_i}\left(G_i\right)C_{n^P_i}(Q_i,\kappa h)\right),\\
    D^{(1)}_{{t}_r}\e_{{t}_r}\left(\cc_p^N (F)\right) & =\sum_{k=1}^p
    \sum_{\substack{|{\bf n}(r)|=k \\ n^P_r>0}} d_k^{\bf n} K_{n^B_r}\left(G_r\right)n^P_r C_{n^P_r-1}(Q_r,\kappa h) \left(\prod_{i<r} K_{n^B_i}\left(G_i\right)C_{n^P_i}(Q_i,\kappa h)\right).
  \end{align*} When $r=0$, we get $\e_{{t}_0}\left(\cc_p^N (F)\right) = d_0 $ and
  we define $D^{(0)}_{\ov{t}_0}\e_{\ov{t}_0}\left(\cc_p^N
   (F)\right)=\frac{1}{\sqrt{h}}d^{\mathbf{e}_1,\mathbf{0_N}}_1$ (which is the limit of
  $D^{(0)}_{t}\e_{t}\left(\cc_p^N (F)\right)$ when $t$ tends to $0$) and  $D^{(1)}_{\ov{t}_0}\e_{\ov{t}_0}\left(\cc_p^N
    (F)\right)=d^{\mathbf{0_N},\mathbf{e}_1}_1$, where $\mathbf{e}_1:=(1,0,\cdots,0)$ of size $N$ and
   $\mathbf{0_N}$ is the vector null of size $N$.
\end{rem}

The following Lemma, similar to Lemma \ref{lem4}, gives
some useful properties of the operator $\cc^N_p.$
\begin{lemme}\label{lem9}
  Let $F$ be a r.v. in $\lp^2(\m F_T)$ and $H$ be in $\hpp^2_T(\rset)$. Then
  \begin{itemize}
  \item $\forall (p,N) \in (\nset^{\star})^2,\; \e(|\cc^N_p(F)|^2) \le
    \e(|\cc_p(F)|^2)\le \e(|F|^2)$.
  \item Let $H$ be in $H^2_T(\rset)$. We deduce from \eqref{eq51} that $\cc^N_p \left(\int_0^T H_s ds\right)=\int_0^T \cc^N_p (H_s) ds$.
  \item For all $F \in \mathbb{D}^{1,2}$, for all $i\in \{0,1\}$ and for all $t \le r$, $D^{(i)}_t \e_r[\cc^N_p(F)]=\e_r[\cc^N_{p-1}(D^{(i)}_t F)]$.
  \end{itemize}
\end{lemme}
Let us end this subsection by some examples.

\begin{ex}[Case $p=2$]
From \eqref{eq:chaos_dec}-\eqref{eq:coef_chaos_dec}, we have
\begin{align*}
  \cc^N_2(F)=&d_0+\sum_{j=1}^N \left(d_1^{j,B}
  K_1(G_j) + d_1^{j,P} C_1(Q_j,\kappa
  h)\right)+\sum_{j=1}^N \left(d_2^{j,B} K_2(G_j)+d_2^{j,P} C_2(Q_j,\kappa
  h)\right)\\
&+
\sum_{j=1}^N\sum_{i=1}^{j-1}\left(d_2^{i,j,B}K_1(G_i)K_1(G_j)+d_2^{i,j,P}C_1(Q_i,\kappa  h)C_1(Q_j,\kappa h)\right) \\ &+\sum_{i=1}^N \sum_{j=1}^N d_2^{i,j,B,P} K_1(G_i)C_1(Q_j,\kappa  h),
\end{align*}
where 
\begin{align*}
  &d_1^{j,B}=\e(FK_1(G_j)),\;d_1^{j,P}=\frac{1}{\kappa
    h}\e(FC_1(Q_j, \kappa h)),\\
  &d_2^{j,B}=2\e(FK_2(G_j)),\;d_2^{j,P}=\frac{1}{2(\kappa
    h)^2}\e(FC_2(Q_j, \kappa h)),\\
  & d_2^{i,j,B}=\e(FK_1(G_i)K_1(G_j)),\;d_2^{i,j,P}=\frac{1}{(\kappa
    h)^2}\e(FC_1(Q_i, \kappa h)C_1(Q_j, \kappa h)),\\
  &d_2^{i,j,B,P}=\frac{1}{\kappa h}\e(FK_1(G_i)C_1(Q_j, \kappa h)).
\end{align*}
Remark \ref{rem1} leads to
\begin{align*}
  \e_{t_r}(\cc^N_2(F))=&d_0+\sum_{j=1}^r \left(d_1^{j,B}
  K_1(G_j) + d_1^{j,P} C_1(Q_j,\kappa
  h)\right)+\sum_{j=1}^r \left(d_2^{j,B} K_2(G_j)+d_2^{j,P} C_2(Q_j,\kappa
  h)\right)\\
&+ \sum_{j=1}^r\sum_{i=1}^{j-1}\left(d_2^{i,j,B}K_1(G_i)K_1(G_j)+d_2^{i,j,P}C_1(\Delta
    N_i,\kappa  h)C_1(Q_j,\kappa h)\right) \\&+\sum_{i=1}^r \sum_{j=1}^r d_2^{i,j,B,P} K_1(G_i)C_1(Q_j,\kappa  h).
\end{align*}

\end{ex}

\section{Numerical scheme}\label{sect:numerical-scheme}

\subsection{Picard's approximation}
 Picard's iterations: $(Y^0,Z^0,U^0)=(0,0,0)$ and for $q\in\nset$,
\begin{equation*} Y^{q+1}_t = \xi + \int_t^T f\left(s,Y^q_s,Z^q_s,U^q_s\right) ds -
  \int_t^T Z^{q+1}_s dB_s-\int_{]t,T]} U^{q+1}_s d\tilde{N}_s,\quad 0\leq t\leq T.
\end{equation*}
It is well-known that the sequence $(Y^q,Z^q,U^q)$ converges exponentially fast towards the solution $(Y,Z,U)$ to BSDE~\eqref{eq:main}.
We write this Picard scheme in a forward way. Let $F^q$ denote $F^q:= \xi +
\int_0^T f\left(s,Y^q_s,Z^q_s,U^q_s\right) ds$. We define
\begin{align}
  &Y^{q+1}_t  = \e\left(F^q \:\Big|\: \m F_t\right) - \int_0^t f\left(s,Y^q_s,Z^q_s,U^q_s\right) ds,\label{eq:Y_q} \\
  &Z^{q+1}_t   =   \e\left(D^{(0)}_t F^q  \:\Big|\:\m F_{t^-}\right),\quad
  U^{q+1}_t =  \e\left(D^{(1)}_t F^q \:\Big|\: \m F_{t^-}\right). \label{eq:ZU_q}
\end{align}

\subsection{Chaos approximation}
Let $(Y^{q,p},Z^{q,p},U^{q,p})$ denote the approximation of $(Y^q,Z^q,U^q)$ built at step
$q$ using a chaos decomposition up to order $p$: $(Y^{0,p},Z^{0,p},U^{0,p})=(0,0,0)$ and 
\begin{align}\label{eq:Y_qp}
  &Y^{q+1,p}_t=\e\left[\cc_p\left(F^{q,p}\right)\:\Big|\: \m F_{t}\right]-\int_0^t
  f\left(s,Y^{q,p}_s,Z^{q,p}_s,U^{q,p}_s\right)ds,\\
  &Z^{q+1,p}_t= \e\left[D_t^{(0)} \cc_p\left(F^{q,p}\right) \:\Big|\:\m F_{t^-}\right], \;\;U^{q+1,p}_t= \e \left[D_t^{(1)} \cc_p\left(F^{q,p}\right) \:\Big|\:\m F_{t^-}\right]\label{eq:ZU_qp},
\end{align}
where $F^{q,p}=\xi+\int_0^T
f\left(s,Y^{q,p}_s,Z^{q,p}_s,U^{q,p}_s \right)ds$.

\subsubsection{Truncation of the basis} The third type of approximation comes
from the truncation of the orthonormal $\lp^2(0,T)$ basis $\{e_i\}_{i \ge 1}$ defined in \eqref{basis-fixed}.  Instead of
considering the whole basis we only keep the first $N$ functions
$\{e_1,\cdots,e_N\}$ to build the chaos decomposition projections $\cc_p^N$. Proposition
\ref{prop2} gives us explicit formulas for $\e_t (\cc_p^N
(F))$, $D^{(0)}_t \e_t (\cc_p^N
(F))$ and $D^{(1)}_t \e_t (\cc_p^N
(F))$. From \eqref{eq:Y_qp} and \eqref{eq:ZU_qp}, we build $(Y^{q,p,N},Z^{q,p,N},U^{q,p,N})_q$ in the following
way : $(Y^{0,p,N},Z^{0,p,N},U^{0,p,N})=(0,0,0)$ and 
\begin{align}
  &Y^{q+1,p,N}_t= \e_t(\cc^N_p (F^{q,p,N}))- \int_0^t
  f\left(s,Y^{q,p,N}_s,Z^{q,p,N}_s,U^{q,p,N}_s \right) ds,\label{eq:Y_qpN}\\
  & Z^{q+1,p,N}_t=D^{(0)}_t(\e_t(\cc^N_p
  (F^{q,p,N}))), \quad U^{q+1,p,N}_t=D^{(1)}_t(\e_t(\cc^N_p (F^{q,p,N})))\label{eq:ZU_qpN},
\end{align}
where  $F^{q,p,N}:= \xi + \int_0^T f(s,Y^{q,p,N}_s,Z^{q,p,N}_s,U^{q,p,N}_s) ds$.

It is not necessary here to use predictable projections of $Z^{q+1,p,N}$ and $U^{q+1,p,N}.$ In fact,  $Z^{q+1,p,N}$ and $U^{q+1,p,N}$ are adapted and c\`adl\`ag,  
 and from their explicit representation  given above one concludes that the predictable projections are the  left-continuous modifications: $\e_{t^-} Z^{q+1,p,N}_t = Z^{q+1,p,N}_{t^-}$ and
$\e_{t^-} U^{q+1,p,N}_t = U^{q+1,p,N}_{t^-}.$ Moreover,  the integral in \eqref{eq:Y_qpN} does not change if one uses  left-continuous modifications. 
  
\subsubsection{Monte Carlo approximation}
\label{sec:MC} Let $F$ denote a r.v. of
$\lp^2(\m F_T)$. In practise, when we are not able to compute exactly $d_0$ and/or
the coefficients  $d^{\bf n}_k$ of the chaos decomposition
\eqref{eq:chaos_dec}-\eqref{eq:coef_chaos_dec} of $F$, we use Monte-Carlo
simulations to approximate them. Let $\{F^m\}_{1\le m \le M}$ be a $M$
i.i.d. sample of $F$ and $\{(G_1^m,Q^m_1), \cdots,(G_N^m,Q_N^m)\}_{1\le m \le M}$ be a $M$
i.i.d. sample of $\{(G_1,Q_1),\cdots,(G_N,Q_N)\}$.

We approximate the expectations of
  \eqref{eq:coef_chaos_dec} by empirical
  means 
\begin{align}\label{d_hat}
    \widehat{d_0}:=\frac{1}{M}\sum_{m=1}^M
    F^m,\;\;\widehat{d^{\bf n}_k}:=\frac{{\bf n}^B!}{{\bf n}^P! (\kappa
      h)^{|{\bf n}^P|}M}\sum_{m=1}^M \left(F^m\prod_{i=1}^N
  K_{n^B_i}\left(G^m_i\right)C_{n^P_i}(Q^m_i,\kappa h)\right).
\end{align}

In the following,  we denote
\begin{align}\label{chaos_dec_MC}
  \cc_p^{N,M} (F)=\widehat{d_0} + \sum_{k=1}^p \sum_{|{\bf n}|=k} \widehat{d_k^{\bf n}}
  \prod_{1\leq i\leq N} K_{n^B_i}(G_i)C_{n_i^P}(Q_i,\kappa h).
\end{align}
$\e_t(\cc_p^{N,M} (F))$ and
$D_t(\e_t(\cc_p^{N,M} (F)))$ denote the conditional expectations obtained in
Proposition \ref{prop2} when $(d_0,\{d^{\bf n}_k\}_{1 \le k \le p,|{\bf n}|=k})$ are
replaced by $(\widehat{d_0},\{\widehat{d^{\bf n}_k})_{1 \le k \le p,|{\bf n}|=k})$ :

\begin{align*}
  &\e_t\left(\cc_p^{N,M} (F)\right)  = \widehat{d_0} + \\
  &\sum_{k=1}^p \sum_{|{\bf n}(r)|=k} \widehat{d_k^{\bf n}}
   \left(\frac{t-\ov{t}_{r-1}}{h}\right)^{\frac{n^B_r}{2}}
   K_{n^B_r}\left(\frac{B_t-B_{\ov{t}_{r-1}}}{\sqrt{t-\ov{t}_{r-1}}}\right)C_{n^P_r}(N_t-N_{\ov{t}_{r-1}},\kappa(t-\ov{t}_{r-1}))\\ 
& \quad\quad\quad\quad\quad\quad\quad\quad\quad\quad\quad\quad\quad\quad\quad\quad\quad\quad\quad\quad \times \underbrace{\left(\prod_{i<r}
       K_{n^B_i}(G_i)C_{n^P_i}(Q_i,\kappa h)\right)}_{:=A_r}, 
 \end{align*}
 \begin{align*}
 &D^{(0)}_t \e_t\left(\cc_p^{N,M}(F)\right)=\\
&= h^{-1/2} \sum_{k=1}^p \sum_{\substack{|{\bf n}(r)|=k \\ n^B_r>0}} \widehat{d_k^{\bf n}}
   \left(\frac{t-\ov{t}_{r-1}}{h}\right)^{\frac{n^B_r-1}{2}}
   K_{n^B_r-1}\left(\frac{B_t-B_{\ov{t}_{r-1}}}{\sqrt{t-\ov{t}_{r-1}}}\right)C_{n^P_r}(N_t-N_{\ov{t}_{r-1}},\kappa(t-\ov{t}_{r-1}))A_r,
 \end{align*}
 \begin{align*}
   &D^{(1)}_t \e_t\left(\cc_p^{N,M}(F)\right)=\\
   &= \sum_{k=1}^p \sum_{\substack{|{\bf n}(r)|=k \\ n^P_r>0}} \widehat{d_k^{\bf n}}
   \left(\frac{t-\ov{t}_{r-1}}{h}\right)^{\frac{n^B_r}{2}}
   K_{n^B_r}\left(\frac{B_t-B_{\ov{t}_{r-1}}}{\sqrt{t-\ov{t}_{r-1}}}\right)n^P_r C_{n^P_r-1}(N_t-N_{\ov{t}_{r-1}},\kappa(t-\ov{t}_{r-1}))A_r.
\end{align*}

 \begin{rem}\label{rem2} As pointed out in \cite[Remark 3.2]{BL_14}, when $M$
  samples of $\cc_p^{N,M} (F)$ are needed, we can either use the same samples
  as the ones used to compute $\widehat{d_0}$ and $\widehat{d_k^{\bf n}}$ or use new ones. In
  the first case, we only require $M$ samples of $F$ and
   $(G_1,\cdots,G_N,Q_1,\cdots,Q_N)$. The coefficients $\widehat{d_k^{\bf n}}$ and $\widehat{d_0}$
  are not independent of $\prod_{1\leq i\leq N} K_{n^B_i}(G_i)C_{n_i^P}(Q_i,\kappa h)$. In this case, the notation
  $\e_t(\cc_p^{N,M} (F))$ introduced above cannot be linked to
  $\e\left(\cc_p^{N,M} (F)|\cf_t\right)$. In the second case, the coefficients
  $\widehat{d_k^{\bf n}}$ and $\widehat{d_0}$ are independent of $\prod_{1\leq i\leq N} K_{n^B_i}(G_i)C_{n_i^P}(Q_i,\kappa h)$ and we have $\e_t\left(\cc_p^{N,M}
    (F)\right)=\e\left(\cc_p^{N,M} (F)|\cf_t\right)$. This second approach
  requires $2M$ samples of $F$ and $(G_1,\cdots,G_N,Q_1,\cdots,Q_N)$. 
    Convergence results are proved when using the second approach. 
  
\end{rem}

We introduce the processes $(Y^{q+1,p,N,M},Z^{q+1,p,N,M},U^{q+1,p,N,M})$,
useful in the following. It corresponds to the approximation of
$(Y^{q+1,p,N},Z^{q+1,p,N},U^{q+1,p,N})$ when we use $\cc_p^{N,M}$ instead of
$\cc_p^N$, i.e. when we use a Monte Carlo procedure to compute the
coefficients $d^{\bf n}_k$.
\begin{align}\label{eq:Y_qpNM}
   &Y^{q+1,p,N,M}_t= \e_t(\cc^{N,M}_p (F^{q,p,N,M}))- \int_0^t
   f\left({\theta}^{q,p,N,M}_s \right) ds,\\
   &Z^{q+1,p,N,M}_t=D^{(0)}_t(\e_t(\cc^{N,M}_p
  (F^{q,p,N,M}))),\;\;U^{q+1,p,N,M}_t=D^{(1)}_t(\e_t(\cc^{N,M}_p
  (F^{q,p,N,M})))\label{eq:ZU_qpNM},
\end{align}
where  $F^{q,p,N,M}:= \xi + \int_0^T
f({\theta}^{q,p,N,M}_s) ds$ and
${\theta}^{q,p,N,M}_s=\left(s,{Y}^{q,p,N,M}_s,{Z}^{q,p,N,M}_s,U^{q,p,N,M}_s \right)$.

\section{Convergence results}\label{sect:conv-results}

We aim at bounding the error between $(Y,Z)$ ~\textemdash~ the solution of
\eqref{eq:main}~\textemdash~ and $(Y^{q,p,N,M}, $ $ Z^{q,p,N,M})$ defined by
\eqref{eq:Y_qpNM}-\eqref{eq:ZU_qpNM}. Before stating the main result of the paper, we introduce
some hypotheses.

\begin{hypo}[Hypothesis $\mathcal{H}_m$]\label{hypo3}
  Let $m \in \nset^*$. We say that $F$ satisfies Hypothesis $\mathcal{H}_m$ if
  $F$ satisfies the two following hypotheses
  \begin{itemize}
  \item $\mathcal{H}_m^1$ : $\forall j  \in \nset^*$ $F \in \mathcal{D}^{m,j}$, i.e. 
    $\norm{F}^j_{m,j} < \infty.$
  \item   $\mathcal{H}_m^2$ : $\forall  j  \in \nset^*$  $\forall l_0, l_1 \in  \nset $ such that 
   $l=l_0+l_1+1 \le m$  there exist two positive constants $\beta_F$ and $k^F_l(j)$ such that  for all
    multi-indices $\alpha=(\alpha_1,\cdots,\alpha_{l_0})\in
    \{0,1\}^{l_0},$  $\gamma=(\gamma_1,\cdots,\gamma_{l_1+1}) \in \{0, 1 \}^{l_1+1}$ and for a.e. $t_i \in [0,T],$  $s_i \in [0,T]$ it holds
  \begin{align*}
  \operatorname*{ess \,sup}_{t_1, \cdots, t_{l_0}}   \operatorname*{ess \,sup}_{s_{i+1}, \cdots, s_{i+l_1} } 
    \e| D^{\alpha}_{t_1,\cdots,t_{l_0}}(D^{\gamma}_{t_i,s_{i+1},\cdots,s_{i+l_1}}
  F-D^{\gamma}_{s_i,\cdots,s_{i+l_1}} F )|^j \le k^{F}_l(j)  \, |t_i-s_i|^{j \beta_{F}}.
  \end{align*}
  In the following, we denote $K^F_m(j) =\max_{l\le m} k^F_l(j)$.
\end{itemize}

\end{hypo}

\begin{rem}\label{rem12}
  If $F$ satisfies $\mathcal{H}_m$,  for all $l\le m$ and for all multi-indices
  $\alpha=(\alpha_1,\cdots,\alpha_l) \in \{0, 1 \}^l$ we have for a.e. $(t_1,\cdots,t_l) \in [0,T]^l$ and $(s_1,\cdots,s_l) \in [0,T]^l$ that
\begin{align}\label{eq31}
  |\e(D^{\alpha}_{t_1,\cdots,t_l}F)-\e(D^{\alpha}_{s_1,\cdots,s_l}F)|\le
 K^F_m(1) (|t_1-s_1|^{\beta_{F}}+\cdots+|t_l-s_l|^{\beta_{F}}).
\end{align}
\end{rem}

 \begin{hypo}[Hypothesis $\mathcal{H}^3_{p,N}$]\label{hypo4}
  Let $(p,N) \in \nset^2$. We say that a r.v. $F$ satisfies $\mathcal{H}^3_{p,N}$ if
  \begin{align*}
    V_{p,N}(F):=\V(F)+\sum_{k=1}^p \sum_{|{\bf n}|=k}\frac{({\bf n}^B)!}{({\bf n}^P)!(\kappa
      h)^{|{\bf n}^P|}} \V\left(F\prod_{i=1}^N
      K_{n_i^B}(G_i)C_{n^P_i}(Q_i,\kappa h)\right) < \infty,
  \end{align*}
 where $\V(\xi)$ denotes the variance of a r.v. $\xi.$
\end{hypo}

\begin{rem}\label{rem9}
  If $F$ is bounded by $K$, we get $ V_{p,N}(F)\le K^2 \sum_{k=0}^p \binom{2N}{k}
    $. Hence every bounded r.v. satisfies $\mathcal{H}^3_{p,N}$.
\end{rem}
This remark ensues from $\e\left(\prod_{i=1}^N
  K_{n_i^B}^2(G_i)C_{n^P_i}^{2}(Q_i,\kappa h)\right)=\frac{({\bf n}^P)!(\kappa h)^{|{\bf n}^P|}}{({\bf n}^B)!}$.

  \begin{rem}\label{rem6} Let $X$ be the $\rset$-valued  solution of
  \begin{align*}
    X_t=x+\int_0^t b(s,X_s)ds +\int_0^t \sigma(s,X_s)dB_s + \int_0^t \gamma(s,X_{s-})d\tilde N_s , \quad t \in [0,T],
  \end{align*}
  where $b,\sigma, \gamma   :[0,T] \times \rset \rightarrow \rset$  are  $C^{0,m}$ functions such that  all partial  derivatives
  w.r.t.~$x$ of order $1 \le k \le m$ are bounded,  and  all   partial  
  derivatives w.r.t.~$x$ of $\sigma$ and $\gamma$  of order $0 \le k \le m$ are
  Hölder continuous  in time (uniformly in $x$) with exponent  $\alpha_{\sigma}$ and $\alpha_ {\gamma},$ 
  respectively.   Then every random variable $\xi$ of type $g(X_T)$ or $g\left (\int_0^T X_s ds \right )$
 with $g \in C^{\infty}_b( \rset)$ satisfies
 $\mathcal{H}_m$   with $\beta_{\xi} =\alpha_\sigma \wedge \alpha_\gamma \wedge \frac{1}{2},$ and 
 $\mathcal{H}^3_{p,N}$ for all $p$ and $N$.  \\
 To prove that $\mathcal{H}_m$  is satisfied one can use 
 \cite[Theorem 3]{Petrou_08}, while  $\mathcal{H}^3_{p,N}$ is implied by Remark \ref{rem9}.)
 We sketch  how to compute $\beta_\xi$ of Hypothesis  $\mathcal{H}^2_m$ for $\xi = g(X_T).$ We have  
\begin{align*}
D_u^{(0)}   X_T&=  \int_u^T \partial_{sp}b(s,X_s)   D_u^{(0)} X_s  ds \\
& + \sigma(u,X_u) +\int_u^T  \partial_{sp} \sigma(s,X_s)  D_u^{(0)} X_s  dB_s \\
&+ \int_u^T  \partial_{sp}  \gamma(s,X_{s-})  D_u^{(0)} X_{s-}  d\tilde N_s
  \end{align*} 
and
\begin{align*}
D_u^{(1)}   X_T&=  \int_u^T (b(s,X_s  + D_u^{(1)} X_s) -   b(s,X_s )) ds \\ 
&+\int_u^T (  \sigma(s,X_s+ D_u^{(1)} X_s) -\sigma(s,X_s) )   dB_s \\
&+ \gamma(u,X_{u-}) + \int_u^T (  \gamma(s,X_{s-} +  D_u^{(1)}  X_{s-})  -  \gamma(s,X_{s-}) )  d\tilde N_s.
  \end{align*} 
In order to show   
\[
\e  | D_{t_1}^{\alpha_1}  g(X_T) - D_{s_1}^{\alpha_1}  g(X_T)    |^j \le K_1^{\xi }(j) |t_1-s_1|^{j \beta_{\xi}}
\] 
notice first that in view of  Remark \ref{rem70} it holds  
\[| D_{t_1}^{\alpha_1}  g(X_T) - D_{s_1}^{\alpha_1}  g(X_T) | \le \| g'\|_\infty |D_{t_1}^{\alpha_1}X_T -  D_{s_1}^{\alpha_1} X_T|. \]
For the estimate of   $\e |D_{t_1}^{\alpha_1}X_T -  D_{s_1}^{\alpha_1} X_T|^j$ we apply   for the integrals w.r.t.~the Brownian motion   the Burkholder-Davis-Gundy inequality,   
and  for the integrals w.r.t.~the compensated Poisson process a Kunita-Watanabe inequality (see \cite[Formula (4.21)]{applebaum_09}).
Finally, similar considerations  as in the proof of Lemma \ref{lem50} and  Gronwall's Lemma  imply  that  $\beta_{\xi} =\alpha_\sigma \wedge \alpha_\gamma \wedge \frac{1}{2}.$ 
The general case can be shown by induction.       
       
\end{rem}

 \begin{thm}\label{thm1} Let $m$ be an integer  s.t. $1 \le m \le p+1$. Assume that
  $\xi$ satisfies $\mathcal{H}_{p+q+1}$ and $\mathcal{H}^3_{p,N}$ and $f\in
  C^{0,p+q+1,p+q+1,p+q+1}_b$. We have
  \begin{align*}
    \norm{(Y-{Y}^{q,p,N,M},&Z-{Z}^{q,p,N,M}_{-},U-{U}^{q,p,N,M}_{ -})}^2_{\lp^2}\\
    &\le
    \frac{A_0}{2^q} + \frac{A_1(q,m)}{ (p+2-m)  \cdots (p+1)}+A_2(q,p) \left(\frac{T}{N}
    \right)^{2\beta_{\xi} \wedge 1}+ \frac{A_3(q,p,N)}{M},
  \end{align*} where $A_0$ is given in Section \ref{sect:erreur_picard}, $A_1$
  is given in Proposition \ref{prop1}, $A_2$ is given in Proposition
  \ref{prop4}, and $A_3$ is given in Proposition
  \ref{prop5}.\\
  If $f\in C^{0,\infty,\infty,\infty}_b$ and $\xi$ satisfies
  $\mathcal{H}_m$ for all $m \in \nset^*$ and $\mathcal{H}^3_{p,N}$ for all $(p,N) \in \nset^2$, we get
  \begin{align*}\lim_{q \rightarrow \infty}\lim_{p \rightarrow
      \infty}\lim_{N \rightarrow
      \infty}\lim_{M \rightarrow
      \infty}\norm{(Y-Y^{q,p,N,M},Z-Z^{q,p,N,M},U-U^{q,p,N,M})}^2_{\lp^2}=0.
  \end{align*}
\end{thm}

 \begin{proof}[Proof of Theorem \ref{thm1}]
  We split the error into $4$ terms :
  \begin{enumerate}
  \item Picard's
    iterations : $\mathcal{E}^q=\norm{(Y-Y^q,Z-Z^q,U-U^q)}^2_{\lp^2}$, where
    $(Y^q,Z^q,U^q)$ is defined by \eqref{eq:Y_q}-\eqref{eq:ZU_q},
  \item the truncation of the chaos decomposition :
    $\mathcal{E}^{q,p}=\norm{(Y^q-Y^{q,p},Z^q-Z^{q,p},U^q-U^{q,p})}^2_{\lp^2}$, where
    $(Y^{q,p},Z^{q,p},U^{q,p})$ is defined by \eqref{eq:Y_qp}-\eqref{eq:ZU_qp},
  \item the
    truncation of the $\lp^2(0,T)$ basis :
    $\mathcal{E}^{q,p,N}=\norm{(Y^{q,p}-Y^{q,p,N},Z^{q,p}-Z^{q,p,N}_-,U^{q,p}-U^{q,p,N}_-)}^2_{\lp^2}$,
    where $(Y^{q,p,N},Z^{q,p,N},U^{q,p,N})$ is defined by \eqref{eq:Y_qpN}-\eqref{eq:ZU_qpN},
  \item the Monte-Carlo approximation to compute the expectations :
    $\mathcal{E}^{q,p,N,M}=\norm{(Y^{q,p,N}-Y^{q,p,N,M},Z^{q,p,N}_--Z^{q,p,N,M}_-,U^{q,p,N}_-
      -U^{q,p,N,M}_-)}^2_{\lp^2}$,
    where $(Y^{q,p,N,M},Z^{q,p,N,M},U^{q,p,N,M})$ is defined by \eqref{eq:Y_qpNM}-\eqref{eq:ZU_qpNM}.
  \end{enumerate}
  We have
\begin{align*}
  \norm{(Y-Y^{q,p,N,M},Z-Z^{q,p,N,M}_-,U-U^{q,p,N,M}_-)}^2_{\lp^2}\le 4(\mathcal{E}^q+\mathcal{E}^{q,p}+\mathcal{E}^{q,p,N}+\mathcal{E}^{q,p,N,M}).
\end{align*} It remains to combine \eqref{erreur_YZq}, Proposition
\ref{prop1}, Proposition \ref{prop4} and Proposition \ref{prop5} to get the first result.
\end{proof}

\subsection{Picard's iterations}\label{sect:erreur_picard}
The first type of error  has already been studied in \cite{TL_94} (see the
proof of Lemma 2.4), we only recall the main result.

From \cite[Lemma 2.4]{TL_94}, we know that under
Hypothesis \ref{hypo1}, the sequence $(Y^q,Z^q,U^q)_q$ defined by
\eqref{eq:Y_q}-\eqref{eq:ZU_q} converges to $(Y,Z,U)$ $d \p
\times dt$ a.e. and in $\spp^2_T(\rset)\times \hpp^2_{T}(\rset) \times \hpp^2_{T}(\rset)$.
Moreover, we have
\begin{align}\label{erreur_YZq}
  \mathcal{E}^q:=\norm{(Y-Y^q,Z-Z^q,U-U^q)}^2_{\lp^2} \le
  \frac{A_0}{2^q},
\end{align}
where the constant $A_0$ depends on $T$, $\norm{\xi}^2$
and on $\norm{f(\cdot,0,0,0)}^2_{\lp^2_{(0,T)}}$.

\subsection{Error due to the truncation of the chaos decomposition}\label{sect:erreur_chaos}
We assume that the integrals are computed exactly, as well as the
expectations. The error is only due to the truncation of the chaos
decomposition $\cc_p$ introduced in \eqref{eq5}.

For the sequel, we also need the following Lemmas. We postpone their proofs to
the Appendix \ref{proof_lem5}.

\begin{lemme}\label{lem5}Let $m \in \nset^*$.
  Assume that $\xi$ satisfies $\mathcal{H}_{m+q}^1$ and $f \in
  \mathcal{C}^{0, m+q,m+q,m+q}_b$. Then $\forall
  q' \le q$, $\forall p \in \nset$, $(Y^{q'},Z^{q'},U^{q'})$ and $(Y^{q',p},Z^{q',p},U^{q',p})$ belong
  to $\mathcal{S}^{m,\infty}$. Moreover,
  \begin{align*}
    \|(Y^{q},Z^{q},U^q)\|^j_{m,j} \le
    C(j,\|\xi\|_{m+q,\frac{(m+q-1)!}{m!}j},(\|\partial^k_{sp}f\|_{\infty})_{k\le
    m+q}).
\end{align*}
\end{lemme}

\begin{lemme}\label{lem5b}
 Assume that $\xi$ satisfies $\mathcal{H}_{p}^1$ and $f \in
  \mathcal{C}^{0,m\vee p,m\vee p,m\vee p}_b$. Then it holds for any $j\ge 1$
\begin{align*}
  \|(Y^{q,p},Z^{q,p},U^{q,p})\|^j_{m,j} \le
  C(p,j,\|\xi\|_{p,1},(\|\partial_{sp}^k f\|_{\infty})_{k\le m\vee p}).
\end{align*}
\end{lemme}

 \begin{prop}\label{prop1} Let $1 \le m \le p+1$. Assume that $\xi$
  satisfies $\mathcal{H}^1_{m+q}$ and $f\in
  C^{0,m+q,m+q,m+q}_b$. We recall $\mathcal{E}^{q,p}=\norm{(Y^q-Y^{q,p},Z^q-Z^{q,p},U^q-U^{q,p})}^2_{\lp^2}$. We get  
  \begin{align}\label{eq7}
    \mathcal{E}^{q+1,p} \le C_1 T(T+1) L_f^2  \mathcal{E}^{q,p} +\frac{K_1(q,m)}{ (p+2-m) \cdots (p+1)},
  \end{align}
   where  $C_1$ is a scalar and the constant $K_1(q,m)$ depends on $T$, $m$,
   $\|\xi\|_{m+q,2\frac{(m+q-1)!}{ (m-1)!}}$ and
  on $(\norm{\partial^k_{sp} f}_{\infty})_{1\le k \le m+q}$.  \smallskip

  Since $\mathcal{E}^{0,p}=0$, we deduce from \eqref{eq7} that $
  \mathcal{E}^{q,p}\le \frac{A_1(q,m)}{ (p+2-m) \cdots (p+1)}$ where $
  A_1(q,m):=
\frac{(C_1T(T+1)L_f^2)^q-1}{C_1T(T+1)L_f^2-1} \times  \max_{1\le l \le q }K_1(l,m)$. Then, $(Y^{p,q},Z^{p,q},U^{q,p})$
  converges to $(Y^q,Z^q,U^q)$ when $p$ tends to $\infty$ in
  $\norm{(\cdot,\cdot,\cdot)}_{\lp^2}$ (see \eqref{norme_L2_YZ} for the definition of
  the norm).
\end{prop}

\begin{rem} We deduce from Proposition \ref{prop1} that for all $T$ and $L_f$,
  we have $ lim_{p\rightarrow \infty}
  \mathcal{E}^{q,p}=0$. When $C_1T(T+1) L_f^2<1$, i.e. for $T$ small enough,  and  if $\xi$ satisfies $\mathcal{H}^1_{\infty}$ and $f
  \in \mathcal{C}_b^{0,\infty,\infty,\infty}$, we
  also get $\ lim_{p\rightarrow \infty} lim_{q\rightarrow \infty}
  \mathcal{E}^{q,p}=0$. \!Indeed, it holds  $\lim_{q\rightarrow \infty} \mathcal{E}^{q,p}\!\le \frac{\sup_j K_1(j,m)}{1-C_1T(T+1)L_f^2}
    \frac{1}{(p+2-m)\cdots(p+1)}$  and $\sup_j K_1(j,m) < \infty$ since from the proof of Proposition \ref{prop1} one concludes that 
    $K_1(j,m)=60(\|\xi\|^2_{D^m}+T\int_0^T \| f(s,Y^j_s,Z^j_s,U^j_s)\|^2_{D^m}
  ds) \le C(T,m,\|\xi\|_{m+j,2\frac{(m+j-1)!}{(m-1)!}},(\| \partial_{sp}^k f
  \|_{k\le m+j}))$. 
   \end{rem}

\begin{proof}[Proof of Proposition \ref{prop1}] In the following, we denote $\Delta Y^{q,p}_t :=
  Y^{q,p}_t-Y^{q}_t$, $\Delta Z^{q,p}_t :=
  Z^{q,p}_t-Z^{q}_t$, $\Delta U^{q,p}_t :=
  U^{q,p}_t-U^{q}_t$ and $\Delta f^{q,p}_t:=f(t,Y^{q,p}_t,Z^{q,p}_t,U^{q,p}_t)-f(t,Y^{q}_t,Z^{q}_t,U^q_t)$.
  Firstly, we deal with $\e[\sup_{0\le t \le T}
  |\Delta Y^{q+1,p}_t|^2]$.
  From \eqref{eq:Y_q} and \eqref{eq:Y_qp} we get
  \begin{align*}
    \Delta Y^{q+1,p}_t=&\e_t[\cc_p(F^{q,p})-F^q]-\int_0^t
    \Delta f^{q,p}_s ds,\\
    =&\e_t[\cc_p(\xi)-\xi]+\e_t\left[\cc_p\left(\int_0^T f(s,Y^{q,p}_s,Z^{q,p}_s,U^{q,p}_s)
      ds\right)-\int_0^T f(s,Y^{q}_s,Z^{q}_s,U^q_s) ds\right] \\
&-\int_0^t
    \Delta f^{q,p}_s ds.
  \end{align*}
  We introduce $\pm \cc_p\left(\int_0^T f(s,Y^{q}_s,Z^{q}_s,U^q_s) ds \right)$ in
  the second conditional expectation. This leads to
  \begin{align*}
    \Delta Y^{q+1,p}_t=&\e_t[\cc_p(\xi)-\xi]+\e_t\left[\cc_p\left(\int_0^T
        \Delta f^{q,p}_s ds \right)\right]-\int_0^t
    \Delta f^{q,p}_s ds\\
    &+\e_t\left[\int_0^T \cc_p
      (f(s,Y^{q}_s,Z^{q}_s,U^q_s)) - f(s,Y^{q}_s,Z^{q}_s,U^q_s) ds\right],
  \end{align*} where we have used the second property of Lemma \ref{lem4} to
  rewrite the third term on the right-hand side (r.h.s. for short).

From the previous equation, we bound $\e[\sup_{0\le t \le T} |\Delta Y^{q+1,p}_t|^2]$ by
using Doob's maximal inequality and the Lipschitz property of $f$
 \begin{align*}
    \norm{\sup_{0\le t \le T} | \Delta Y^{q+1,p}_t| }_2 &\le 2
     \norm{  \cc_p(\xi)-\xi}_2+  2\left \| \cc_p\left(\int_0^T \Delta f^{q,p}_s ds \right) \right \|_2 \\
   & + 2 \int_0^T \norm{\cc_p(f(s,Y^{q}_s,Z^{q}_s,U^q_s))- f(s,Y^{q}_s,Z^{q}_s,U^q_s) }_2 ds \\ 
    &+ L_f \int_0^T
    \norm{|\Delta Y^{q,p}_s|+|\Delta Z^{q,p}_s|+|\Delta U^{q,p}_s|}_2 ds.
  \end{align*}
 To bound the second term on the r.h.s.  of the previous inequality, we use the first
 property of Lemma \ref{lem4} and the Lipschitz property of $f$. Then, we
 bring together this term with the last one to get

 \begin{align}\label{eq50}
   \norm{\sup_{0\le t \le T} | \Delta Y^{q+1,p}_t| }_2 &\le 2
   \norm{  \cc_p(\xi)-\xi}_2 + 2 \int_0^T \norm{\cc_p(f(s,Y^{q}_s,Z^{q}_s,U^q_s))- f(s,Y^{q}_s,Z^{q}_s,U^q_s) }_2 ds \nonumber \\ 
   &+3 L_f \int_0^T
   \norm{|\Delta Y^{q,p}_s|+|\Delta Z^{q,p}_s|+|\Delta U^{q,p}_s|}_2 ds.
 \end{align}

    Let us now upper bound $  \int_0^T \norm{\Delta Z^{q+1,p}_s}_2^2 ds+  \kappa \int_0^T\norm{ \Delta U^{q+1,p}_s}^2_2 ds$. To do
    so, we use the Itô isometry $ \int_0^T \norm{\Delta Z^{q+1,p}_s}_2^2 ds+ \kappa \int_0^T\norm{ \Delta U^{q+1,p}_s}^2_2 ds= 
    \norm{\int_0^T \Delta Z^{q+1,p}_s dB_s+ \Delta U^{q+1,p}_s d \tilde{N}_s}_2^2.$

Using the Definitions
      \eqref{eq:ZU_q}-\eqref{eq:ZU_qp} of $(Z^{q+1}_t,U^{q+1}_t)$ and
      $(Z^{q+1,p}_t,U^{q+1,p}_t)$ and the Clark-Ocone Formula (see
      \cite[Theorem 1.8]{LSUV_02}) leads to
      \begin{align*}
       \int_0^T
       \Delta Z^{q+1,p}_s dB_s+\int_0^T
       \Delta U^{q+1,p}_s d\tilde{N}_s&=F^q-\e(F^q)-(\cc_p(F^{q,p})-\e(\cc_p(F^{q,p}))),\\
       &=Y^{q+1}_T+\int_0^T f(s,Y_s^q,Z^q_s,U^q_s)ds-Y^{q+1}_0\\
       &-\left(Y^{q+1,p}_T+\int_0^T f(s,Y_s^{q,p},Z^{q,p}_s,U^{q,p}_s)ds-Y^{q+1,p}_0\right).
     \end{align*} Rearranging this summation makes appear
     $\Delta Y^{q+1,p}_T - (\Delta Y^{q+1,p}_0)$. We get

 \begin{align}\label{eq60}
 \int_0^T \norm{\Delta Z^{q+1,p}_s}_2^2 ds+ \kappa \int_0^T\norm{ \Delta U^{q+1,p}_s}^2_2 ds 
&\le 4 \norm{\sup_{0\le t \le T} | \Delta Y^{q+1,p}_t| }_2^2 \nonumber \\
&+2L_f^2 \left (\int_0^T
     \norm{\Delta Y^{q,p}_s}_2+ \norm{\Delta Z^{q,p}_s}_2+ \norm{\Delta U^{q,p}_s}_2 ds \right )^2.
 \end{align}

Since $\left (\int_0^T \norm{\Delta Y^{q,p}_s}_2+ \norm{\Delta Z^{q,p}_s}_2+
  \norm{\Delta U^{q,p}_s}_2 ds \right )^2 \le
\frac{3(1+\kappa)}{\kappa}T(T+1) \mathcal{E}^{q,p}$, by computing $
5\times {\eqref{eq50}}^2 +\eqref{eq60}$  we obtain
\begin{align*}
    \mathcal{E}^{q+1,p} &\le 60\norm{ \cc_p(\xi)-\xi}_2^2
    +60 T \int_0^T \norm{\cc_p(f(s,Y^{q}_s,Z^{q}_s,U^q_s))- f(s,Y^{q}_s,Z^{q}_s,U^q_s) }_2^2 ds \nonumber \\
    &+137  \frac{3(1+\kappa)}{\kappa}  T(T+1)L_f^2  \mathcal{E}^{q,p}.
      \end{align*} 

Since $\xi$ and $f(s,Y^q_s,Z^q_s,U^q_s)$ belong to $\mathbb{D}^{m,2}$ ($\xi$ satisfies $\mathcal{H}^1_{m+q}$, $f \in C^{0,m+q,m+q,m+q}_b$ and
$(Y^q,Z^q,U^q) \in \mathcal{S}^{m,\infty}$ (see Lemma \ref{lem5})), Lemma \ref{lem2} gives
\begin{align*}
  \mathcal{E}^{q+1,p} \le& \frac{60 \norm{\xi}^2_{D^{m}}}{(p+2-m)\cdots(p+1)}
  +\frac{60 T}{(p+2-m)\cdots(p+1)}\int_0^T\norm{ f(s,Y_s^q,Z^q_s,U^q_s)}^2_{D^{m}} ds 
  \\&+ \frac{411(1+\kappa)}{\kappa} T(T+1)L_f^2  \mathcal{E}^{q,p}.
\end{align*}

Since $\int_0^T\!\norm{f(s,Y_s^q,Z^q_s,U^q_s)}^2_{D^{m}} ds$ is
bounded by $C(T,m,(\norm{\partial^k_{sp} f}_{\infty})_{k \le
  m},\!\norm{(Y^q,Z^q,U^q)}^{2m}_{m,2m})$ (see \eqref{eq26}, in the
proof of Lemma \ref{lem5}), Lemma
\ref{lem5} gives the result.

\end{proof}

\subsection{Error due to the truncation of the basis }\label{sect:erreur_basis}

Fix $N \in \nset^*$ and put $h=\frac{T}{N}.$ 
Use $\{p_0, p_1\}= \{ \ind_{\{0\}},  \frac{1}{\sqrt{\kappa }} \ind_{\{1\}}  \}$ as orthonormal basis of $L^2(\{0,1\}, 2^{\{0,1\}}, \delta_0 + \kappa \delta_1)$ and 
fix an orthonormal basis $\{e_k\}_{k \in \nset^*} $ for $L^2([0,T], \m B ([0,T]),  \lambda)$ such that  $\ov{t}_i=ih$ for $ i=0,1,\ldots,N$ and
 \[              e_i = \frac{1}{\sqrt{h}}\ind_{]\ov{t}_{i-1},\ov{t}_i]}(t), \quad 1 \le i \le N.
\]
\begin{lemme} \label{basis-truncation}
Assume $F  = \e [F] + \sum_{n=1}^\infty I_n(g_n) \in  \lp^2(\m F_T)$ satisfies
\eqref{eq31} with $m=p$. Then 
\begin{align*}
  \e|   (\mathcal{C}_p^N -\mathcal{C}_p)(F)|^2 \le  \ov{K_p^F} \bigg ( \frac{T}{N} \bigg)^{2 \beta_F} \sum_{i=1}^{p} i^2 \frac{T^i}{i!} \le \ov{K_p^F}
  \bigg ( \frac{T}{N} \bigg)^{2 \beta_F}  T(1+T) e^T.
\end{align*}
where $\ov{K_p^F}:=\sum_{j=1}^p (K_j^F)^2$ (with $K_j^F:= K_j^F(1)$ from \eqref{eq31}).
\end{lemme}

We refer to Section \ref{sec:proof_basis-truncation} for a proof of Lemma
\ref{basis-truncation}.

\begin{lemme}\label{lem50}
   Assume $\xi$ satisfies $\mathcal{H}_{p}$ (i.e. Hypothesis \ref{hypo3}) and $f \in
  C^{0, p ,  p,  p}_b$. Then, for all integers  $q\ge0$, $I_{q,p}:=\int_0^T
  f(s,Y^{q,p}_s,Z^{q,p}_s,U^{q,p}_s) ds$ satisfies $\mathcal{H}_p$ so that by Remark \ref{rem12} for all $1 \le r \le p$ and 
  multi-indices $\mathbf{i}_r\in \{0,1\}^r$ and  for a.e. $(t_1,\cdots,t_r) \in [0,T]^r$ and $(s_1,\cdots,s_r) \in [0,T]^r$  we have 
\begin{align*}
  |\e(D^{\mathbf{i}_r}_{t_1,\cdots,t_r}I_{q,p})-\e(D^{\mathbf{i}_r}_{s_1,\cdots,s_r}I_{q,p})|\le
  K^{I_{q,p}}_r(|t_1-s_1|^{\beta_{I_{q,p}}}+\cdots+|t_r-s_r|^{\beta_{I_{q,p}}}),
\end{align*} where $\beta_{I_{q,p}}=\frac{1}{2}\wedge \beta_{\xi},$ and  the constant $K^{I_{q,p}}_r$ depends
on   $ K^{\xi}_r$, $\|\xi\|_{p,1},$   $T$ and on
$(\norm{\partial^k_{sp} f}_{\infty})_{1\le k \le  p}$. 
\end{lemme}

We refer to \ref{proof_lem50} for the proof of Lemma \ref{lem50}.

\begin{prop}\label{prop4} Assume that  $\xi$ satisfies $\mathcal{H}_{p}$ and $f\in
  C^{0,p,p,p}_b$. We recall
   $\mathcal{E}^{q,p,N}:=\norm{(Y^{q,p}-Y^{q,p,N},Z^{q,p}-Z^{q,p,N}_{-}, U^{q,p}-U^{q,p,N}_{-})}^2_{\lp^2}$.
  We get
  \begin{align}\label{eq8}
    \mathcal{E}^{q+1,p,N} \le C_2 T(T+1) L_f^2  \mathcal{E}^{q,p,N}
    +K_2(q,p)\left(\frac{T}{N} \right)^{1 \wedge 2\beta_\xi},
  \end{align}
  where $C_2$\! is a scalar and the constant $K_2(q,p)$\! depends
  on  $\ov{ K^{\xi}_p}, T,\|\xi\|_{p, 1}$\! and on $(\norm{\partial^k_{sp}
    f}_{\infty})_{1\le k \le p}$.  \\
 Since $\mathcal{E}^{0,p,N}=0$, we deduce from \eqref{eq8} that $
  \mathcal{E}^{q,p,N}\le A_2(q,p) \left(\frac{T}{N} \right)^{1 \wedge 2\beta_\xi}$, where $A_2(q,p):=
  K_2(q,p)T(T+1)e^T\frac{(C_2T(T+1)L_f^2)^q-1}{C_2T(T+1)L_f^2-1}.$ Then, $(Y^{p,q,N},Z^{p,q,N}_{-},U^{q,p,N}_{-})$
  converges to $(Y^{q,p},Z^{q,p},U^{q,p})$ in
  $\norm{(\cdot,\cdot, \cdot)}_{\lp^2}$ when $N$ tends to $\infty.$ 
\end{prop}

\begin{proof}[Proof of Proposition \ref{prop4}]  In the following, we denote $$\Delta Y^{q,p,N}_t :=
  Y^{q,p,N}_t-Y^{q,p}_t,$$ $$\Delta Z^{q,p,N}_t :=
 Z^{q,p,N}_{t^-}-Z^{q,p}_t, \quad \Delta U^{q,p,N}_t :=
  U^{q,p,N}_{t^-}-U^{q,p}_t,$$ and $$\Delta f^{q,p,N}_t:=f(t,Y^{q,p,N}_t,Z^{q,p,N}_{t^-},U^{q,p,N}_{t^-}) 
  -f(t,Y^{q,p}_t,Z^{q,p}_t, U^{q,p}_t).$$
  Firstly, we deal with $\norm{\sup_{0\le t \le T}
  |\Delta Y^{q+1,p,N}_t|}_2$.
  From \eqref{eq:Y_qp} and \eqref{eq:Y_qpN} we get
  \begin{align*}
    \Delta Y^{q+1,p,N}_t=\e_t[\cc_p^N(F^{q,p,N})-\cc_p(F^{q,p})]+\int_0^t
    \Delta f^{q,p,N}_s ds.
  \end{align*} By using the second property of Lemma \ref{lem9}, by following
  the same steps as in the proof of Proposition \ref{prop1} and by introducing
  $\pm \cc_p^N (\int_0^T f(s,Y^{q,p}_s,Z^{q,p}_s,U^{q,p}_s) ds)$, one gets
   \begin{align*}
     \norm{\sup_{0\le t \le T} | \Delta Y^{q+1,p,N}_t|}_2&\le 
    2 \norm{\cc_p^N(\xi)-\cc_p(\xi)}_2 +  2 \left \| \cc_p^N\left(\int_0^T \Delta f^{q,p}_s ds\right ) \right\|_2\\
    &+  2 \left \|(\cc_p^N-\cc_p)\left (\int_0^T(f(s,Y^{q,p}_s,Z^{q,p}_s,U^{q,p}_s) ds  \right ) \right\|_2 \\
   &+  L_f \int_0^T
    \norm{ |\Delta Y^{q,p,N}_s|+|\Delta Z^{q,p,N}_s|+|\Delta U^{q,p,N}_s| }_2 ds.\notag
  \end{align*}
  It remains to apply the first property of Lemma \ref{lem9} to get
 \begin{align} \label{eq10}
     \norm{\sup_{0\le t \le T} | \Delta Y^{q+1,p,N}_t|}_2 &\le 
    2 \norm{\cc_p^N(\xi)-\cc_p(\xi)}_2 +    2 \left \|(\cc_p^N-\cc_p)\left (\int_0^T(f(s,Y^{q,p}_s,Z^{q,p}_s,U^{q,p}_s) ds  \right ) \right\|_2  \notag \\
  & +  3L_f \int_0^T
    \norm{ |\Delta Y^{q,p,N}_s|+|\Delta Z^{q,p,N}_s|+|\Delta U^{q,p,N}_s| }_2 ds.
  \end{align}

    Let us now upper bound $ \int_0^T \norm{\Delta Z^{q+1,p,N}_s}_2^2 ds+  \kappa \int_0^T\norm{ \Delta U^{q+1,p,N}_s}^2_2 ds.$ 

Following the same steps as in the proof of Proposition \ref{prop1},
one gets
\begin{align}\label{eq9}
& \int_0^T \norm{\Delta Z^{q+1,p,N}_s}_2^2 ds+ \kappa \int_0^T\norm{ \Delta U^{q+1,p,N}_s}^2_2 ds \nonumber \\
&\le 4 \norm{\sup_{0\le t \le T} | \Delta Y^{q+1,p,N}_t| }_2^2 
+2L_f^2 \left (\int_0^T
     \norm{\Delta Y^{q,p,N}_s}_2+ \norm{\Delta Z^{q,p,N}_s}_2+ \norm{\Delta U^{q,p,N}_s}_2 ds \right )^2.
 \end{align} 

 Adding $5\times\eqref{eq10}^2$ and \eqref{eq9} gives
\begin{align*}
    \mathcal{E}^{q+1,p,N} &\le 60\norm{(\cc_p^N - \cc_p)(\xi)  }_2^2 
   +60\left \|(\cc_p^N-\cc_p)\left (\int_0^T(f(s,Y^{q,p}_s,Z^{q,p}_s,U^{q,p}_s) ds  \right ) \right\|^2_2 \nonumber \\
         &+\frac{411(1+\kappa)}{\kappa}  T(T+1)L_f^2  \mathcal{E}^{q,p,N}.
      \end{align*}

Since $\xi$ and $I_{q,p}$ satisfy \eqref{eq31} (see Remark \ref{rem12} and Lemma \ref{lem50}), Lemma
  \ref{basis-truncation} gives
  \begin{align*}
    \mathcal{E}^{q+1,p,N} \le  60\left(\frac{T}{N}\right)^{2 \beta_\xi 
      \wedge 1} T (T+1) e^T \left((\ov{K^{\xi}_p})^2+(\ov{K^{I_{q,p}}_p})^2
    \right)+  \frac{411(1+\kappa)}{\kappa}T(T+1)L_f^2  \mathcal{E}^{q,p,N},
  \end{align*}
  and \eqref{eq8} follows. 
\end{proof}

\subsection{Error due to the Monte-Carlo approximation}\label{sect:MC}

We are now interested in bounding the error between $(Y^{q,p,N},Z^{q,p,N}_{
  -},U^{q,p,N}_{ -})$
defined by \eqref{eq:Y_qpN}-\eqref{eq:ZU_qpN} and $(Y^{q,p,N,M},Z^{q,p,N,M}_{-},U^{q,p,N,M}_{-})$ defined by
\eqref{eq:Y_qpNM}-\eqref{eq:ZU_qpNM}.  $\cc_p^{N,M}$ is defined by \eqref{d_hat} and
\eqref{chaos_dec_MC}. In this Section, we assume that the coefficients
$\hat{d^{\bf n}_k}$ are independent of the vector $(G_1,\cdots,G_N)$, which
corresponds to the second approach proposed in Remark \ref{rem2}.

Before giving an upper bound for the error, we recall the following Lemma,
which measures the error between
$\cc_p^N$ and $\cc_p^{N,M}$ for a r.v. satisfying $\mathcal{H}^3_{p,N}$ (see
Hypothesis \ref{hypo4}).

\begin{lemme}\label{lem7}
  Let $F$ be a r.v. satisfying Hypothesis $\mathcal{H}^3_{p,N}$. We have
  \begin{align*}
    \e(|(\cc_p^N-\cc_p^{N,M})(F)|^2) = \frac{1}{M}V_{p,N}(F).
  \end{align*}
  Moreover, we have $ \e(|\cc_p^{N,M}(F)|^2) \le \e(|F|^2)+\frac{1}{M}V_{p,N}(F)$.
\end{lemme}
We refer to Section \ref{sec:proof_lem7} for the proof of the Lemma.

\begin{prop}\label{prop5} Let $\xi$ satisfy Hypothesis $\mathcal{H}^3_{p,N}$\! and $f$ be a bounded function. Let  
$\mathcal{E}^{q,p,N,M}\!:=\norm{(Y^{q,p,N}-Y^{q,p,N,M},Z^{q,p,N}-Z^{q,p,N,M},U^{q,p,N}-U^{q,p,N,M})}^2_{\lp^2}$.
  We get  
  \begin{align*}
    \mathcal{E}^{q+1,p,N,M} \le C_3 T(T+1) L_f^2  \mathcal{E}^{q,p,N,M}
    +\frac{K_3(q,p,N)}{M},
  \end{align*}
  where $C_3$ is a scalar and the constant $K_3(q,p,N):=  C_4 \left(V_{p,N}(\xi)+T^2
    \|f\|^2_{\infty} \sum_{k=0}^p \binom{2N}{k} \right)$ for some $C_4 >0$.\\
  Since $\mathcal{E}^{0,p,N,M}=0$, we deduce from the previous inequality that $
  \mathcal{E}^{q,p,N,M}\le \frac{A_3(q,p,N)}{M} $, where $A_3(q,p,N):=
  K_3(q,p,N)\frac{(C_3T(T+1)L_f^2)^q-1}{C_3T(T+1)L_f^2-1}$. Then, $(Y^{p,q,N,M},Z^{p,q,N,M},U^{q,p,N,M})$
  converges to $(Y^{q,p,N},Z^{q,p,N},U^{q,p,N})$ in
  $\norm{(\cdot,\cdot,\cdot)}_{\lp^2}$ when $M$ tends to $\infty$ .
\end{prop}

The proof of Proposition \ref{prop5} is the same as the proof of
\cite[Proposition 4.17]{BL_14}, except that we consider jumps. The jump part
is treated as in \eqref{eq9}.

 \section{Implementation}\label{sect:implementation}
\subsection{Pseudo-code of the Algorithm}\label{section:algo}
In this section, we describe in
detail the algorithm. We aim at computing $M$
trajectories of an approximation of $(Y,Z,U)$ on the grid
$\mathcal{T}=\{\overline{t}_i=i\frac{T}{N}, i=0,\cdots,N \}$. Starting from
$(Y^{0,p,N,M},Z^{0,p,N,M},U^{0,p,N,M})=(0,0,0)$, \eqref{eq:Y_qpNM}-\eqref{eq:ZU_qpNM} enable to get
$(Y^{q,p,N,M},Z^{q,p,N,M},U^{q,p,N,M})$ for each Picard's iteration $q$ on
$\mathcal{T}$. In practice, we discretize the integral $ \int_0^t
  f\left({\theta}^{q,p,N,M}_s \right) ds$ which leads to approximated values of
$(Y^{q,p,N,M}, $ $Z^{q,p,N,M},U^{q,p,N,M})$ computed on a grid.
Let us introduce
$(\overline{Y}^{q+1,p,N,M}_{\overline{t}_i},\overline{Z}^{q+1,p,N,M}_{\overline{t}_i},\overline{U}^{q+1,p,N,M}_{\overline{t}_i})_{1\le i \le
  N}$, defined by $$(\overline{Y}^{0,p,N,M},\overline{Z}^{0,p,N,M},\overline{U}^{0,p,N,M})=(0,0,0)$$ and
for all $q \ge 0$
     \begin{align}\label{eq:YZ_qpNM_barre}
    &\overline{Y}^{q+1,p,N,M}_{\overline{t}_i}= \e_{\overline{t}_i}(\cc^{N,M}_p (\overline{F}^{q,p,N,M}))-h \sum_{j=1}^{i}
    f\left(\overline{t}_j,\overline{Y}^{q,p,N,M}_{\overline{t}_j},\overline{Z}^{q,p,N,M}_{\overline{t}_j},\overline{U}^{q,p,N,M}_{\overline{t}_j}\right),\notag\\
    &\overline{Z}^{q+1,p,N,M}_{\overline{t}_i}=D^{(0)}_{\overline{t}_i}(\e_{\overline{t}_i}(\cc^{N,M}_p(\overline{F}^{q,p,N,M}))),\notag\\
    &\overline{U}^{q+1,p,N,M}_{\overline{t}_i}=D^{(1)}_{\overline{t}_i}(\e_{\overline{t}_i}(\cc^{N,M}_p(\overline{F}^{q,p,N,M}))),
  \end{align}
where  $\overline{F}^{q,p,N,M}:= \xi + h\sum_{i=1}^{N}
f(\overline{t}_i,\overline{Y}^{q,p,N,M}_{\overline{t}_i},\overline{Z}^{q,p,N,M}_{\overline{t}_i},\overline{U}^{q,p,N,M}_{\overline{t}_i})$.

 \begin{rem}
  Instead of studying the error between $(Y,Z,U)$ and $\theta^{q,p,N,M}:=(Y^{q,p,N,M}, $ $Z^{q,p,N,M},U^{q,p,N,M})$
  we could have studied the error between $(Y,Z,U)$ and
  $\overline{\theta}^{q,p,N,M}:=(\overline{Y}^{q,p,N,M},$ $\overline{Z}^{q,p,N,M},\overline{U}^{q,p,N,M})$. The
  main difference between $\theta^{q,p,N,M}$ and
  $\overline{\theta}^{q,p,N,M}$ is that we consider a discrete sum in the
  implemented scheme. In that case, the scheme of the proof is the same, and in order to get the same convergence rate we
  just need to add another assumption: $f$ has to be Hölder-$(\beta_\xi \wedge \frac{1}{2})$ in time.
\end{rem}

Here are the notations we use in the algorithm.

 \begin{itemize}
 \item $q$: index of Picard's iteration
 \item $K_{it}$: number of Picard's iterations
 \item $M$: number of Monte--Carlo samples
 \item $N$: number of time steps used for the discretization of $Y$ and $Z$
 \item $p$: order of the chaos decomposition 
 \item $\bY^q \in \mathcal{M}_{N+1,M}(\rset)$ represents $M$ paths of $\overline{Y}^{q,p,N,M}$
   computed on the grid $\mathcal{T}$.
 \item  $\bZ^q \in
   \mathcal{M}_{N+1,M}(\rset)$ (resp. $\bU^q \in
   \mathcal{M}_{N+1,M}(\rset)$) represents $M$ paths of
   $\overline{Z}^{q,p,N,M}$ (resp. $\overline{U}^{q,p,N,M}$) computed on
   the grid $\mathcal{T}$.
\end{itemize}

Since $\xi \in \lp^2(\m F_T)$, $\xi $ can be written as a measurable function
of $\{B_t,N_t\}_{t \le T}$. Then, one gets one sample of $\xi$ from one sample of
$((G_1,Q_1),\cdots,(G_N,Q_N))$ (where $G_i$ represents
$\frac{B_{\overline{t}_i}-B_{\overline{t}_{i-1}}}{\sqrt{h}}$ and $Q_i$
represents $N_{\overline{t}_i}-N_{\overline{t}_{i-1}}$).
 
\begin{algorithm}[H]
  \caption{Iterative algorithm}\label{algodetail}
  \begin{algorithmic}[1] \State Pick at random $N\times M$ values of standard
    Gaussian r.v., stored in $\bG$, and $N\times M$ values of Poisson r.v. of
    parameter $\kappa h$ stored in $\bQ$.  \State Using $\bG$
    and $\bQ$, compute $\{\xi^m\}_{0 \le m \le M-1}$.  \State $\bY^0 \equiv 0$,
    $\bZ^0 \equiv 0$, $\bU^0 \equiv 0$.  \For{$q=0:K_{it}-1$} \label{Kit} \For
    {$m=0:M-1$}\label{loopF} \State Compute
    $(F^q)^m=\xi^m+h\sum_{i=1}^{N}f(\overline{t}_i,(\bY^q)_{i,m},(\bZ^q)_{i,m},(\bU^q)_{i,m})$
    \EndFor
    \State Compute the vector $\bf{d}=(\widehat{d_0},\{\widehat{d^n_k}\}_{1\le
      k \le p, |n|=k})$ of the chaos decomposition of $F^q$\label{coef}
    \State
    $\widehat{d_0}:=\frac{1}{M}\sum_{m=0}^{M-1}
    (F^q)^m,\;\;\widehat{d^{\bf n}_k}=\frac{{\bf n}^B!}{{\bf n}^P!(\kappa
      h)^{|{\bf n}^P|}M}\sum_{m=0}^{M-1} (F^q)^m \prod_{i=1}^N
    K_{n_i^B}(G_i^m)C_{n_i^P}(Q_i^m,\kappa h)$
    \For {$j=1:N$}\label{lineN}
    \For {$m=0:M-1$}\label{lineM}
    \State Compute 
    $(\e_{\overline{t}_j}(\cc_p^{N,M} F^q))^m$,
    $(D^{(0)}_{\overline{t}_j}(\e_{\overline{t}_j}(\cc_p^{N,M} F^q)))^m$,$(D^{(1)}_{\overline{t}_j}(\e_{\overline{t}_j}(\cc_p^{N,M} F^q)))^m$\label{loopEt}
    \State $(\bY^{q+1})_{j,m}=(\e_{\overline{t}_j}(\cc_p^{N,M}
    F^q))^m-h\sum_{i=1}^{j}f(\overline{t}_i,(\bY^q)_{i,m},(\bZ^q)_{i,m},(\bU^q)_{i,m})$\label{loopYZ}
    \State
    $(\bZ^{q+1})_{j,m}=(D^{(0)}_{\overline{t}_j}(\e_{\overline{t}_j}(\cc_p^{N,M}
    F^q)))^m$
    \State
    $(\bU^{q+1})_{j,m}=(D^{(1)}_{\overline{t}_j}(\e_{\overline{t}_j}(\cc_p^{N,M} F^q)))^m$
    \EndFor
    \EndFor
    \EndFor
    \State Return $(\bY^{K_{it}})_{0,:}=\widehat{d}_0$,
    $(\bZ^{K_{it}})_{0,:}=\frac{1}{\sqrt{h}}\widehat{d}^{\mathbf{e}_1,\mathbf{0_N}}_1$ and $(\bU^{K_{it}})_{0,:}=\widehat{d}^{\mathbf{0_N},\mathbf{e}_1}_1$ 
\end{algorithmic}
\end{algorithm}

Let us now deal with the complexity of the algorithm :\\
For each $q$:
\begin{itemize}
  \item the computation of the vector $F^q$ (loop line \ref{loopF}) requires
    $O(M\times N)$ computations,
  \item the computation of the vector $\mathbf{d}$ (line \ref{coef}) requires $O(M\times p \times
    N^p)$ computations, and the computation of each coefficient  requires
    $O(M\times p)$
    computations,
    \item for each $N$ and $M$ (lines \ref{lineN}-\ref{lineM})
    \begin{itemize}
    \item the computation of $(\e_{\overline{t}_j}(\cc_p^{N,M} F^q))^m$, of
      $(D^{(0)}_{\overline{t}_j}(\e_{\overline{t}_j}(\cc_p^{N,M} F^q)))^m$ and
      of \\ $(D^{(1)}_{\overline{t}_j}(\e_{\overline{t}_j}(\cc_p^{N,M} F^q)))^m$ (line \ref{loopEt}) requires
      $O( p\times N^p)$ computations
    \item the computation of $(\bY^{q+1})_{j,m}$ (loop line \ref{loopYZ})
      requires $O(N)$ computations, the computation of
      $((\bZ^{q+1})^l_{j,m})$ and $((\bU^{q+1})^l_{j,m})$ requires $O(1)$ computations.
    \end{itemize}
  \end{itemize}
  The complexity of the algorithm is then $O(K_{it}\times M
  \times p\times N^{p+1})$.
 
\begin{rem}Given the complexity $C_0$ of the algorithm, we can choose the parameters
  $p,q,N$ and $M$ such that they minimize the error $ \frac{A_0}{2^q} + \frac{A_1(q,p)}{(p+1)!}+A_2(q,p) \left(\frac{T}{N}
    \right)^{a}+\frac{A_3(q,p,N)}{M}$, where $a:=2\beta_{\xi} \wedge 1$. This boils down to solving the
    following constrained minimization problem
    \begin{align*}
      \min_{q,p,N,M \mbox{ s.t. }qpMN^{p+1}=C_0} \left(  \frac{1}{2^q} +
        \frac{C^q}{(p+1)!}+\frac{C^q}{N^a}+\frac{C^q N^p}{M}\right).
    \end{align*}
    The Karush-Kuhn-Tucker theorem gives $M\sim\frac{2p}{a}(p+1)^{(p+\frac{3}{2})(\frac{p}{a}+1)}$, $N\sim (p+1)^\frac{p+\frac{3}{2}}{a}$, $q\sim \frac{1}{\ln(2C)}p \ln (p+1)$ and $p$ such that
    $(p+1)^{(2p+3)(1+\frac{p}{a})} p^3 \ln (p+1) \sim a\log(2C)C_0$.
  \end{rem}
  
\subsection{Numerical Examples}
\subsubsection{First example}
The following example is borrowed from \cite{LMT_14}.
We consider a Poisson process $N$ with $\kappa=1$ and the following BSDE
\begin{align*}
  &dY_t=-c U_t dt +Z_t dB_t +U_t(dN_t-dt),\\
  &\xi=N_T.
\end{align*}
The explicit solution is given by
\begin{align*}
  (Y_t,Z_t,U_t)=(N_t+(1+c)(T-t),0,1).
\end{align*}
Figure \ref{fig1} represents the evolution of $(Y^{q,p,N,M}_0,Z^{q,p,N,M}_0,U^{q,p,N,M}_0)$ with respect to
$M$ when $q=5$, $p=2$ and $N=20$.

\begin{figure}[H]
  \begin{center}
    \includegraphics[width=15cm]{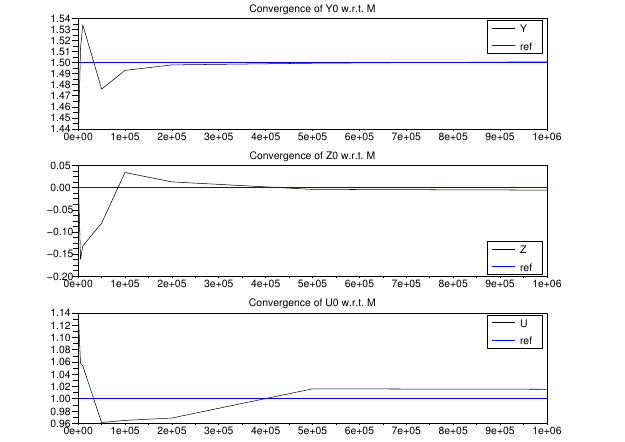}
    \caption{Evolution of $(Y^{q,p,N,M}_0,Z^{q,p,N,M}_0,U^{q,p,N,M}_0)$ with
      respect to $M$ when $p=2$, $N=20$, $q=5$, $c=0.5$, $T=1$}\label{fig1}
  \end{center}
\end{figure}

Table \ref{tab1} gives the computational
time needed by the algorithm with this choice for $q$, $p$, $N$ and for
different values of $M$. We notice from Figure \ref{fig1} that the value of
$(Y^{q,p,N,M}_0,Z^{q,p,N,M}_0,$ $U^{q,p,N,M}_0)$ is close to the true
solution from $M=2\times10^5$. When $M=2\times10^5$, the CPU time is about $1$ minute,
which is quite small.

\begin{table}[htbp]
  \begin{center}
    \begin{tabular}{|c||c|c|c|c|c|c|c|c|}
      \hline
      M & $10^3$ & $5\times 10^3$ & $10^4$ & $5\times 10^4$ & $10^5$ & $2\times 10^5$ & $5\times
      10^5$ & $10^6$\\
      \hline
      CPU time (in $s$) & 0.253 & 1.277 & 2.567 & 13.24 & 26.81 & 56.91 & 142.75 & 283.65 \\
      \hline
    \end{tabular}
  \end{center}
  \caption{CPU time w.r.t. $M$ when $p=2$, $N=20$, $q=5$, $c=0.5$, $T=1$}\label{tab1}
\end{table}

\subsubsection{Second example}

We consider now the following BSDE
\begin{align*}
  &dY_t=-(\alpha Y_t + \beta Z_t+ \gamma U_t) dt +Z_t dB_t +U_t d\tilde{N}_t,\\
  &\xi=\exp(aT+b B_T + c N_T).
\end{align*}
 The explicit solution is given by
\begin{align*}
& Y_t= e^{aT + bB_t + c N_t }
e^{(\alpha + \frac{(b  + \beta)^2-\beta^2  }{2} )(T-t)+ (e^c-1) (\kappa + \gamma) (T-t)},\\ 
&  Z_t = \e_{t^-}(D_t^0 Y_t) = b Y_{t^-},\;\; U_t =\e_{t-} (D_t^1 Y_t) = (e^c-1) Y_{t-}
\end{align*}

We choose $\alpha=\beta=0.3$, $\gamma=0.2$, $a=-0.1$, $b=0.1$, $c=0.2$,
$\kappa=3$ and $T=2$. For these values, we get
$(Y_0,Z_0,U_0)=(6.599,0.66,1.4612)$. For $M=4\times 10^5$, $p=2$, $N=50$ and
$q=10$, we get $(Y^{q,p,N,M}_0, Z^{q,p,N,M}_0, U^{q,p,N,M}_0)=
(6.560,0.56,1.294)$. We plot one path of
$(Y^{q,p,N,M}_t,Y_t)_{t\le T}$, $(Z^{q,p,N,M}_t,Z_t)_{t\le T}$ and $(U^{q,p,N,M}_t,U_t)_{t\le T}$ in
Figures \ref{fig2}, \ref{fig3}, \ref{fig4} with $M=4\times 10^5$, $p=2$, $N=50$ and $q=10$.
\begin{figure}[H]
  \begin{center}
  \vspace{-1em}
  \includegraphics[width=12cm]{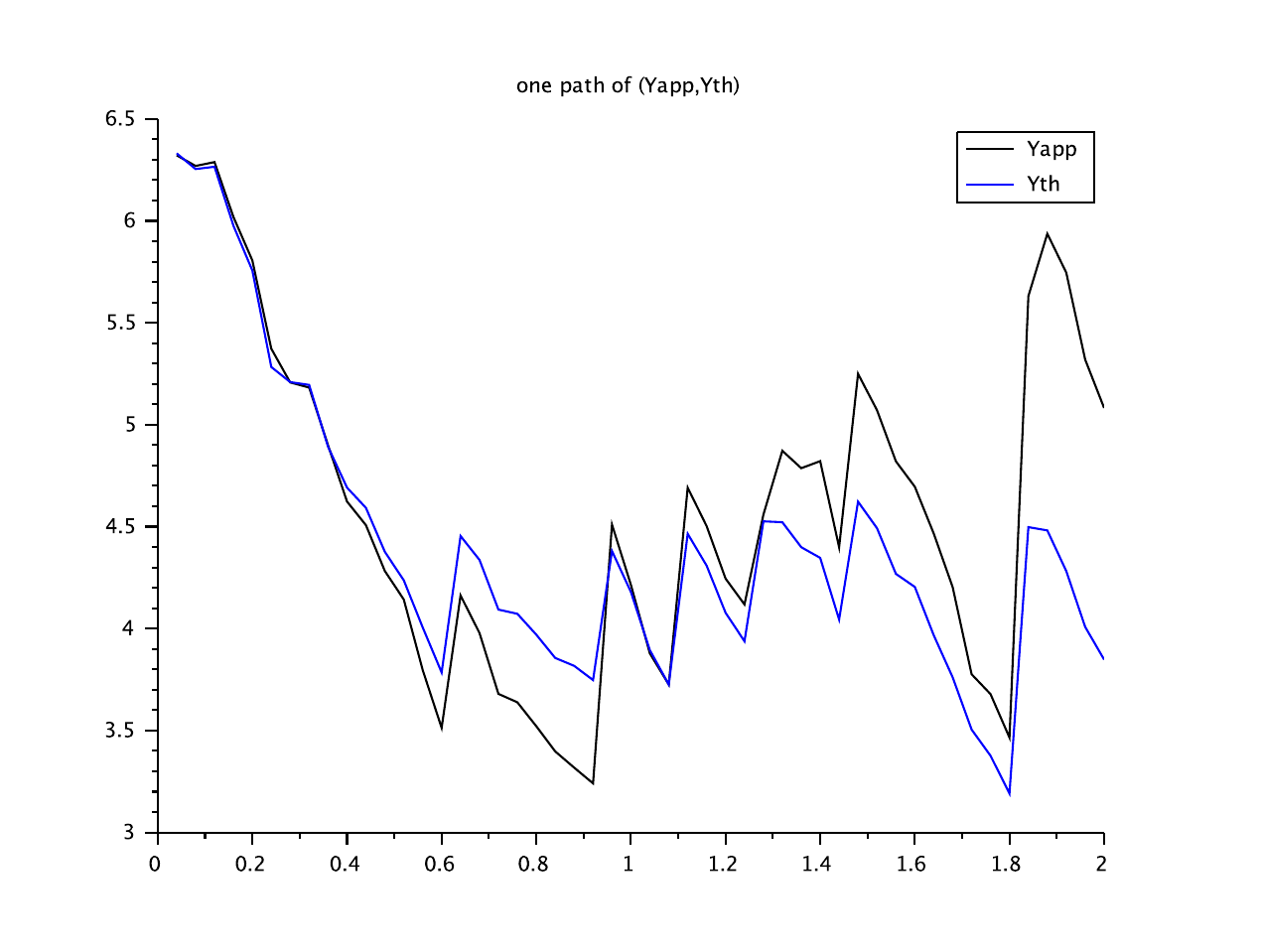}
      \caption{One path of $(Y^{q,p,N,M},Y)$}\label{fig2}
  \end{center}
\end{figure}
\begin{figure}[H]
  \begin{center}
  \includegraphics[width=12cm]{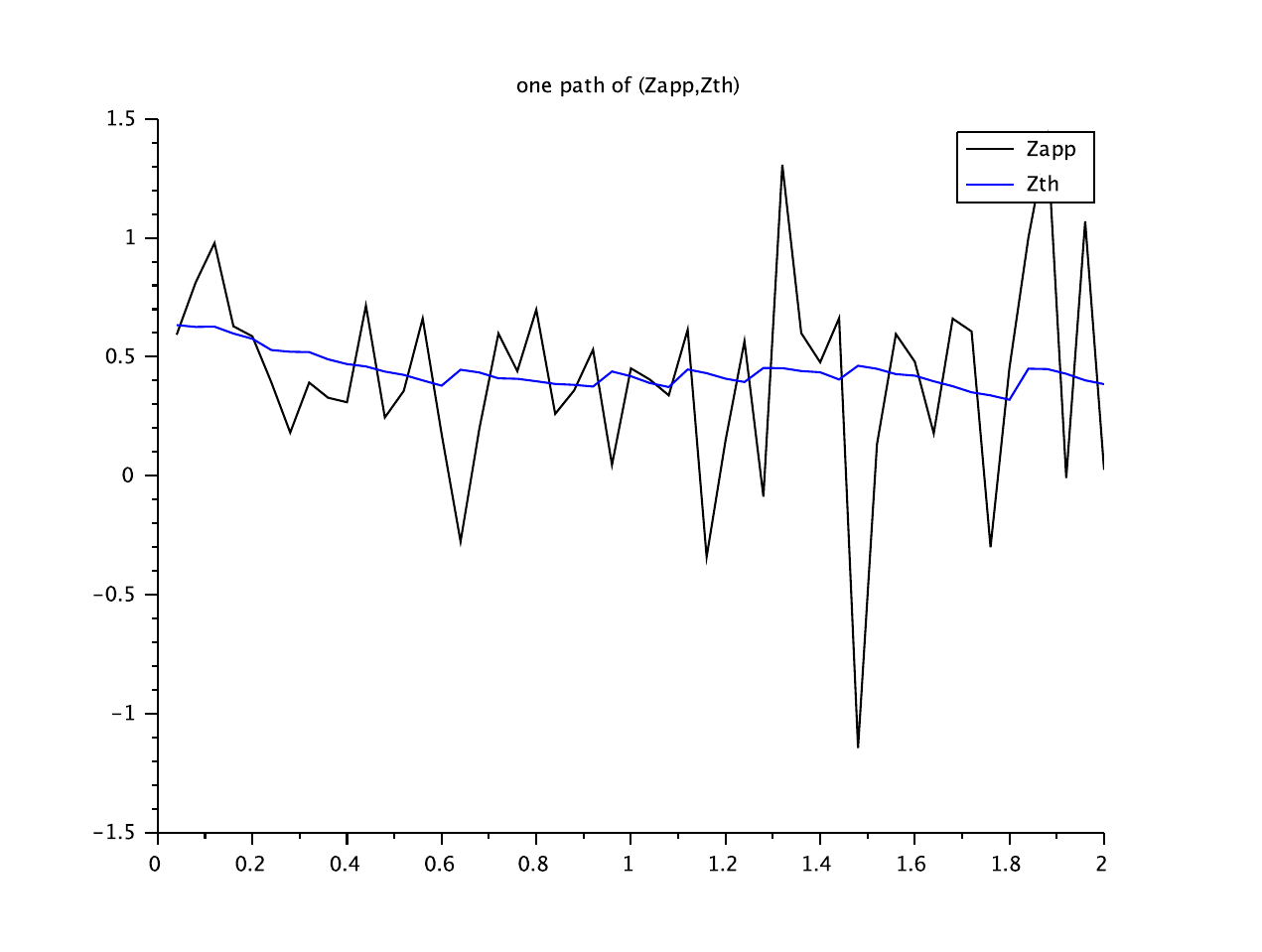}
    \caption{One path of $(Z^{q,p,N,M},Z)$}\label{fig3}
  \end{center}
\end{figure}

\begin{figure}[H]
  \begin{center}
  \includegraphics[width=12cm]{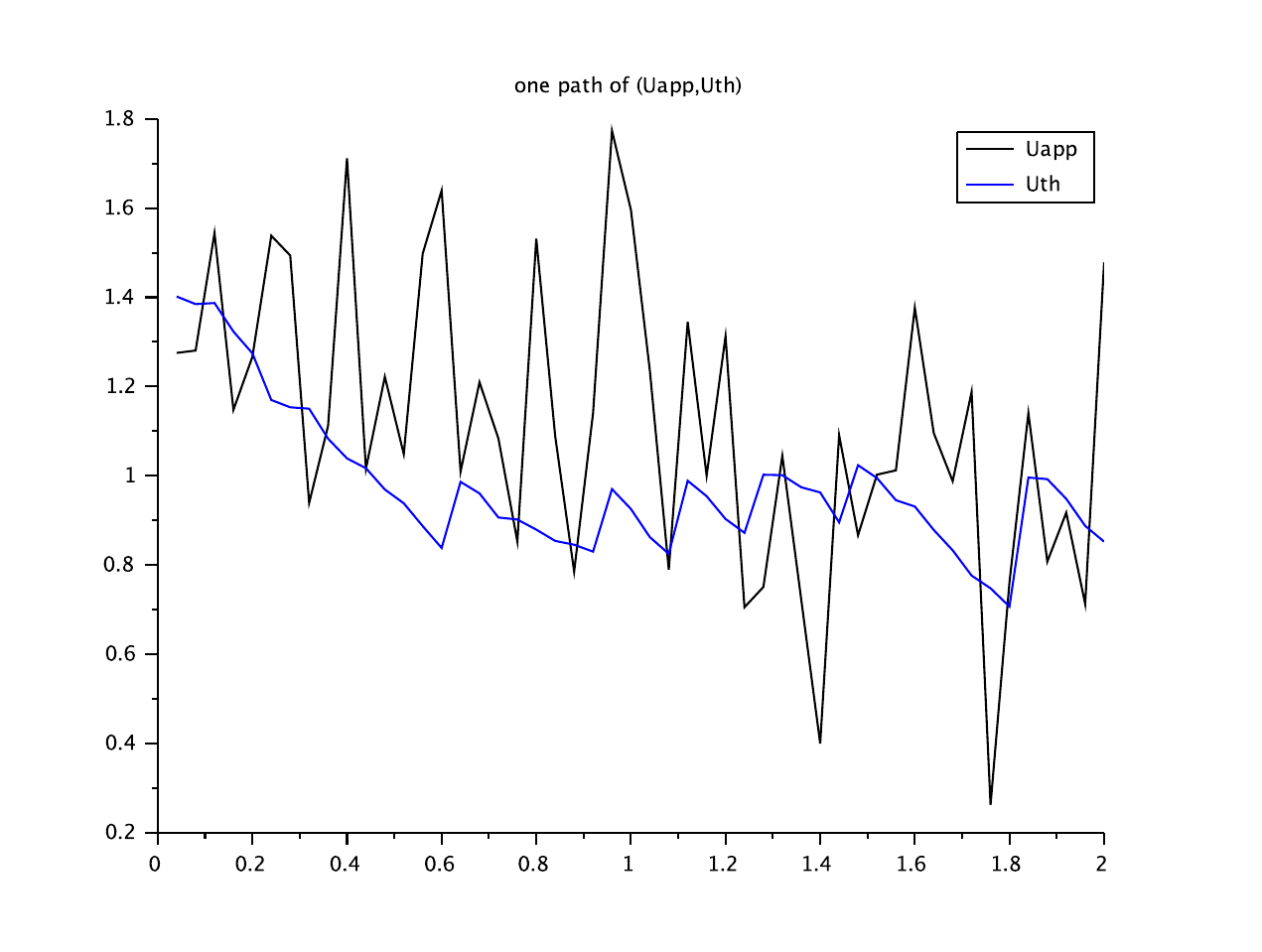}
    \caption{One path of $(U^{q,p,N,M},U)$}\label{fig4}
  \end{center}
\end{figure}

\appendix
\section{Technical results}\label{sec:appendix}
  
\subsection{Proof of Lemmas \ref{lem5} and \ref{lem5b}}\label{proof_lem5}
  \subsubsection{Proof of Lemma \ref{lem5}}
Let $\tilde{t}_l:=\max\{t_1,\cdots,t_l\}$. First, we prove by induction that $\forall q' \le q$, $(Y^{q'},Z^{q'},U^{q'})$ belongs to $\mathcal{S}^{m,\infty}$,
i.e. $\forall j \ge 2$
\begin{align*}
 \|(Y^{q'},Z^{q'},U^{q'})\|^j_{m,j}=\sum_{1\le l \le m} \sum_{ \mathbf{i}_l} \operatorname*{ess \,sup}_{t_1, \cdots, t_l}&\left\{
    \e[\sup_{\tilde{t}_l\le r \le T} |D^{\mathbf{i}_l}_{t_1, \cdots, t_l} Y^{q'}_r|^j] +\int_{\tilde{t}_l}^T \e[|D^{\mathbf{i}_l}_{t_1, \cdots, t_l} Z^{q'}_r|^j]
    dr\right.\\
  &\left.+\int_{\tilde{t}_l}^T \e[|D^{\mathbf{i}_l}_{t_1, \cdots, t_l} U^{q'}_r|^j]
    dr\right\}< \infty.
\end{align*}
Let $r \ge \tilde{t}_l$. Using \eqref{eq:Y_q} gives
\begin{align*}
  D^{\mathbf{i}_l}_{t_1, \cdots, t_l} Y^{q'}_r=\e_r[ D^{\mathbf{i}_l}_{t_1, \cdots, t_l}
  F^{q'-1}]-\!\int_{\tilde{t}_l}^r \!D^{\mathbf{i}_l}_{t_1, \cdots, t_l}  f(\theta^{q'-1}_s) ds, \,
  \mbox{where $\theta^{q'-1}_s:=(s,Y^{q'-1}_s,Z^{q'-1}_s,U^{q'-1}_s)$.}
\end{align*}

Using the Definition of $F^{q'-1}$ and applying Doob's inequality leads to
\begin{align*}
  \e[\sup_{t_l\le r \le T} |D^{\mathbf{i}_l}_{t_1, \cdots, t_l} Y^{q'}_r|^j] \le C(j)\left( \e[|D^{\mathbf{i}_l}_{t_1, \cdots, t_l}
  \xi|^j]+\e\left(\int_{\tilde{t}_l}^T |D^{\mathbf{i}_l}_{t_1, \cdots, t_l}
      f(\theta^{q'-1}_s)|^j ds\right)\right),
\end{align*}
where $C(j)$ is a generic constant depending also on $T$. Analyzing the
outcome of the repeated Malliavin derivative where for $D^{(0)}_tf(\theta^{q'-1}_s)$ the chain rule holds while 
$$D^{(1)}_tf(\theta^{q'-1}_s) = f(s,Y^{q'-1}_s +D^{(1)}_tY^{q'-1}_s ,Z^{q'-1}_s +D^{ (1)}_tZ^{q'-1}_s ,U^{q'-1}_s+D^{ (1)}_tU^{q'-1}_s ) -  f(\theta^{q'-1}_s) $$  (see, for example, \cite[Lemma 3.2]{GS_13}), 
one can see that the term 
$|D^{\mathbf{i}_l}_{t_1, \cdots, t_l} f(\theta^{q'-1}_s)|$ is bounded
by a sum of
terms of type
\begin{align*}
\left (\sum_{k=1}^{l_0+l_1+l_2} \|\partial_{sp}^k f\|_{\infty} \right )
 \left |  D^{\mathbf{k}_0}_{\mathbf{t}_0} Y_s^{q'-1} \right |
\left |D^{\mathbf{k}_1}_{\mathbf{t}_1} Z_s^{q'-1}\right |
\left |D^{\mathbf{k}_2}_{\mathbf{t}_2} U_s^{q'-1} \right | ,
\end{align*}
 where $\mathbf{k}_j \in \{0,1\}^{l_j}$ are
vectors of size $l_j$ and 
$l_0+l_1+l_2 \le l.$ Then, Hölder's inequality gives
\begin{align}\label{eq26}
 \e\left(\int_{\tilde{t}_l}^T |D^{\mathbf{i}_l}_{t_1, \cdots, t_l}
   f(\theta^{q'-1}_s)|^j ds\right) \le C\left  (\sum_{k=1}^{l} \|\partial^k_{sp}f\|^j_{\infty} \right )\|(Y^{q'-1},Z^{q'-1},U^{q'-1})\|^{lj}_{l,lj}
  \end{align}
  and
  \begin{align}\label{eq25}
     \sum_{1\le l \le m} \sum_{\mathbf{i}_l \in \{0,1\}^l}  & \operatorname*{ess \,sup}_{t_1, \cdots, t_l} \,
    \e[ \sup_{t_l\le r \le T} |D^{\mathbf{i}_l}_{t_1, \cdots, t_l}
    Y^{q'}_r|^j]  \notag \\
    &\le
    C(j)\left(\|\xi\|^j_{m,j}+\sum_{l=1}^m \left (\sum_{k=1}^l \|\partial^k_{sp}f\|^j_{\infty} \right) \|(Y^{q'-1},Z^{q'-1},U^{q'-1})\|^{lj}_{l,lj}\right).
  \end{align}

  From \eqref{eq:ZU_q}, we get $D^{\mathbf{i}_l}_{t_1, \cdots, t_l}
  Z^{q'}_r=\e_r[D^{(i_1,\cdots,i_l,0)}_{t_1, \cdots, t_l,r} \xi +\int_r^T D^{(i_1,\cdots,i_l,0)}_{t_1,
    \cdots, t_l,r} f(\theta^{q'-1}_u) du]\ind_{\{r\ge \tilde{t}_l\}}$. Then
  \begin{align*}
    \int_{\tilde{t}_l}^T \e[|D^{\mathbf{i}_l}_{t_1, \cdots, t_l} & Z^{q'}_r|^j]
    dr \\ & \le C(j)\left(\int_{\tilde{t}_l}^T \e[|D^{(i_1,\cdots,i_l,0)}_{t_1, \cdots, t_l,r}
      \xi|^j]dr+\int_{\tilde{t}_l}^T \e\left(\left|\int_r^T D^{(i_1,\cdots,i_l,0)}_{t_1,
    \cdots, t_l,r} f(\theta^{q'-1}_u) du\right|^j\right)dr\right).
\end{align*}

  Using \eqref{eq26} yields
  \begin{align*}
     \sum_{1\le l \le m}\sum_{\mathbf{i}_l \in \{0,1\}^l} & \operatorname*{ess \,sup}_{t_1, \cdots, t_l} \,  \int_{\tilde{t}_l}^T \e[|D^{\mathbf{i}_l}_{t_1, \cdots, t_l} Z^{q'}_r|^j]
     dr \\
     &\le C(j)\left( \|\xi\|^j_{m+1,j}+\sum_{l=1}^m \left (\sum_{k=1}^{ l+1} \|\partial^k_{sp}f\|^j_{\infty} \right) \|(Y^{q'-1},Z^{q'-1},U^{q'-1})\|^{(l+1)j}_{(l+1),(l+1)j}\right).
  \end{align*}
  The same type of result holds for $\int_{\tilde{t}_l}^T \e[|D^{\mathbf{i}_l}_{t_1, \cdots, t_l} U^{q'}_r|^j]dr $.
  Combining this result with \eqref{eq25} gives
  $$\|(Y^{q'},Z^{q'},U^{q'})\|^j_{m,j} \le
  C(j)\!\left(\!\|\xi\|^j_{ m+1,j}+ \left (\sum_{k=1}^{m+1}
  \|\partial^k_{sp}f\|^j_{\infty} \!\right )\sum_{l=1}^{m}\| (Y^{q'-1},Z^{q'-1},U^{q'-1})\|^{(l+1)j}_{(l+1),(l+1)j}\right).
  $$
  Iterating this inequality yields the result.

 \subsubsection{Proof of Lemma \ref{lem5b}}
   We prove it by induction on $q.$ Let $r\ge \tilde{t}_l\!:=\max\{t_1,\cdots,t_l\}$ and   $\theta^{q,p}_s:=(s,Y^{q,p}_s,Z^{q,p}_s,U^{q,p}_s).$ From \eqref{eq:Y_qp} we get that
  \begin{align*}
    D^{\mathbf{i}_l}_{t_1, \cdots, t_l} Y^{q+1,p}_r&=\e_r[ D^{\mathbf{i}_l}_{t_1, \cdots, t_l}
 \cc_p (F^{q,p})]-\int_{\tilde{t}_l}^r D^{\mathbf{i}_l}_{t_1, \cdots, t_l}  f(\theta^{q,p}_s) ds \\
 &=\e_r[ 
 \cc_{p-l} (D^{\mathbf{i}_l}_{t_1, \cdots, t_l}  F^{q,p})]\ind_{\{l\le p\}}-\int_{\tilde{t}_l}^r D^{\mathbf{i}_l}_{t_1, \cdots, t_l}  f(\theta^{q,p}_s) ds,
  \end{align*}
  where we have used Lemma \ref{lem4} to get the second equality.
  Applying Doob's maximal inequality leads to
\begin{align}\label{eq30}
  \e[\sup_{\tilde{t}_l\le r \le T} |D^{\mathbf{i}_l}_{t_1, \cdots, t_l} Y^{q+1,p}_r|^j] \le & C(j)\bigg ( \e[|\cc_{p-l} (D^{\mathbf{i}_l}_{t_1, \cdots, t_l}  F^{q,p})|^j]\ind_{\{l\le p\}} \nonumber  \\ 
  & \quad \quad \quad+\e\left(\int_{\tilde{t}_l}^T |D^{\mathbf{i}_l}_{t_1, \cdots, t_l}
      f(\theta^{q,p}_s)|^j ds\right)\bigg),
\end{align} where $C(j)$ is a generic constant depending also on $T$. Let us
first deal with the first term of the r.h.s. of \eqref{eq30}, we assume $l \le
p$. Following Proposition \ref{chaos-expansion}, we know that
$F^{q,p}=\sum_{n=0}^{\infty} I_n(g_n)$. Then
\begin{align*}
  D^{\mathbf{i}_l}_{t_1, \cdots, t_l} F^{q,p}&=\sum_{n=l}^{\infty}n(n-1)\cdots(n-l+1)I_{n-l}(g_n(*,z_1,\cdots,z_l)),
  \end{align*} 
  with $z_k =(t_k,i_k)$ and
\begin{align*}  
  \cc_{p-l} (D^{\mathbf{i}_l}_{t_1, \cdots, t_l}  F^{q,p})&=\sum_{n=l}^p
  n(n-1)\cdots(n-l+1)I_{n-l}(g_n(*,z_1,\cdots,z_l)),\\
  &=\sum_{n=0}^{p-l} \frac{(n+l)!}{n!} I_n(g_{n+l}(*,z_1,\cdots,z_l)).
\end{align*} Let us denote $\hat g_{n_i}(*):= g_{n_i+l}(*,z_1,\cdots,z_l)$. 
From Proposition \ref{E-of-chaos-products} we get
\begin{align} \label{E-of-product-of-chaos}
&  \e[|\cc_{p-l} (D^{\mathbf{i}_l}_{t_1, \cdots, t_l}
  F^{q,p})|^j] \nonumber \\
  &=\sum_{n_1,\cdots,n_j=0}^{p-l} \e(I_{n_1}(\hat g_{n_1})\cdots
  I_{n_j}(\hat g_{n_j})) \frac{(n_1+l)!}{n_1!}\cdots \frac{(n_j+l)!}{n_j!}  \nonumber\\
  &=\sum_{n_1,\cdots,n_j=0}^{p-l} \!\!\! \frac{(n_1+l)!}{n_1!}\cdots \frac{(n_j+l)!}{n_j!} \!\!\!
   \sum_{J^B \in [n]} \,\, \sum_{(\tau,\sigma) \in \Pi_{=2,\ge 2}(J^B; n_1,\ldots,n_j)}
 \kappa^{|\sigma|}  \int_{[0,T]^{|\tau|+|\sigma|}} \bigg ( \bigotimes_{i=1}^j \hat g_{n_i} \bigg )_{\tau \cup\sigma}  d\lambda^{|\tau|+|\sigma|} \nonumber \\
   & \le \sum_{n_1,\cdots,n_j=0}^{p-l} \frac{(n_1+l)!}{n_1!}\cdots \frac{(n_j+l)!}{n_j!} \prod_{i=1}^j
  \|g_{n_i+l}\|_{\infty} 
   \sum_{J^B \in [n]} \,\, \sum_{(\tau,\sigma) \in \Pi_{=2,\ge 2}(J^B; n_1,\ldots,n_j)}
 \kappa^{|\sigma|} \, T^{|\tau|+|\sigma|}. 
 \end{align} Thanks to the assumptions on $f$ and $\xi$ and induction
hypothesis, we have $F^{q,p} \in \mathbb{D}^{p,2}$. Then,
\eqref{kernel-Malliavin-rel} gives that
$g_{n_i+l}(z_1,\cdots,z_{n_i+l})=\frac{1}{(n_i+l)!}\e(D^{(i_1,\cdots,i_{n_i+l})}_{t_1,
  \cdots, t_{n_i+l}}( F^{q,p}))$, then $\|g_{n_i+l}\|_{\infty} \le
\frac{1}{(n_i+l)!}  \|F^{q,p}\|_{n_i+l,1}$. Since $n_i \le p-l$, we get
$\|g_{n_i+l}\|_{\infty} \le \frac{1}{(n_i+l)!}  \|F^{q,p}\|_{p,1}$.  Then

\begin{align}\label{eq32}
  \e[|\cc_{p-l} (D^{\mathbf{i}_l}_{t_1, \cdots, t_l}
  F^{q,p})|^j]&\le (\|F^{q,p}\|_{p,1})^j\sum_{n_1,\cdots,n_j=0}^{p-l}  
   \sum_{J^B \in [n]} \,\, \sum_{(\tau,\sigma) \in \Pi_{=2,\ge 2}(J^B; n_1,\ldots,n_j)}
 \kappa^{|\sigma|} \, T^{|\tau|+|\sigma|} \nonumber\\
& \le C(p,j)
  (\|F^{q,p}\|_{p,1})^j.
  \end{align} 

We have $\|F^{q,p}\|_{p,1}=\sum_{l\le p} \sum_{\mathbf{i}_l \in \{0,1\}^l}
\operatorname*{ess \,sup}_{t_1, \cdots, t_l} \e(|D_{t_1,\cdots,t_l}^{\mathbf{i}_l}
F^{q,p}|)$ where
\begin{align*}
  \e(|D_{t_1,\cdots,t_l}^{\mathbf{i}_l}
F^{q,p}|)\le \e(|D_{t_1,\cdots,t_l}^{\mathbf{i}_l} \xi|)+\e(\int_{\tilde{t}_l}^T |D_{t_1,\cdots,t_l}^{\mathbf{i}_l} f(\theta_s^{q,p})|ds).
\end{align*}
By using \eqref{eq26}, we get $\e\left(\int_{\tilde{t}_l}^T
  |D^{\mathbf{i}_l}_{t_1, \cdots, t_l} f(\theta^{q,p}_s)| ds\right)\le
C \left (\sum_{k=1}^l
\|\partial^k_{sp}f\|_{\infty} \right)\|(Y^{q,p},Z^{q,p},U^{q,p})\|^{l}_{l,l}$. Then
\begin{align}
  \|F^{q,p}\|_{p,1} & \le \|\xi \|_{p,1}+\sum_{l\le p} C\left (\sum_{k=1}^l
  \|\partial^k_{sp}f\|_{\infty} \right )\|(Y^{q,p},Z^{q,p},U^{q,p})\|^{l}_{l,l},\label{eq31c}\\
 \|F^{q,p}\|_{p,1}^j &\le C(p,j)\left( \|\xi \|^j_{p,1} +\sum_{l \le p}C\left(\sum_{k=1}^l \|\partial^k_{sp}f\|^j_{\infty}\right) \|(Y^{q,p},Z^{q,p},U^{q,p})\|^{lj}_{l,lj}\right).\label{eq31b}
\end{align}
  
Let us now deal with the second term of the r.h.s. of \eqref{eq30}. By using
\eqref{eq26}, we get
\begin{align}\label{eq33}
 \e\left(\int_{\tilde{t}_l}^T |D^{\mathbf{i}_l}_{t_1, \cdots, t_l}
      f(\theta^{q,p}_s)|^j ds\right)\le C\left(\sum_{k=1}^l \|\partial^k_{sp}f\|^j_{\infty}\right)\|(Y^{q,p},Z^{q,p},U^{q,p})\|^{lj}_{l,lj}.
  \end{align}
  Combining \eqref{eq32}, \eqref{eq31b}, \eqref{eq33} and \eqref{eq30} yields
  \begin{align*}
    & \e[\sup_{\tilde{t}_l\le r \le T} |D^{\mathbf{i}_l}_{t_1, \cdots, t_l}
    Y^{q+1,p}_r|^j]\\& \le  C(p,j)\left(\|\xi \|^j_{p,1} +\sum_{l \le p} C\left(\sum_{k=1}^l
      \|\partial^k_{sp}f\|^j_{\infty}\right)
      \|(Y^{q,p},Z^{q,p},U^{q,p})\|^{lj}_{l,lj}\right)\ind_{\{l\le p\}}\\
    &+C\left (\sum_{k=1}^l \|\partial^k_{sp}f\|^j_{\infty} \right)\|(Y^{q,p},Z^{q,p},U^{q,p})\|^{lj}_{l,lj},
  \end{align*}
   and
  \begin{align}\label{eq34}
 &    \sum_{1\le l \le m} \sum_{\mathbf{i}_l \in \{0,1\}^l} \operatorname*{ess \,sup}_{t_1, \cdots, t_l}
    \e[\sup_{\tilde{t}_l\le r \le T} |D^{\mathbf{i}_l}_{t_1, \cdots, t_l} Y^{q+1,p}_r|^j] \nonumber \\&\le
    C(p,j)\left(\|\xi\|^j_{p,1}+\sum_{l=1}^{m\vee p} C \left (\sum_{k=1}^l \|\partial^k_{sp}f\|^j_{\infty} \right) \|(Y^{q,p},Z^{q,p},U^{q,p})\|^{lj}_{l,lj}\right).
  \end{align}
  From \eqref{eq:ZU_qp} we get 
  $$D^{(i_1,\cdots,i_l)}_{t_1, \cdots, t_l}
  Z^{q+1,p}_r=\e_r[D^{(i_1,\cdots,i_l,0)}_{t_1, \cdots, t_l,r}
  \cc_p(F^{q,p})]=\e_r[\cc_{p-l-1} (D^{(i_1,\cdots,i_l,0)}_{t_1, \cdots,
    t_l,r} F^{q,p})]\ind_{\{l \le p-1\}}\ind_{\{r\ge \tilde{t}_l\}}.$$ 
    Then
  \begin{align*}
    \int_{\tilde{t}_l}^T \e[|D^{\mathbf{i}_l}_{t_1, \cdots, t_l} Z^{q+1,p}_r|^j]
    dr \le C\left(\int_{\tilde{t}_l}^T \e[|\cc_{p-l-1} (D^{(i_1,\cdots,i_l,0)}_{t_1, \cdots,
        t_l,r} F^{q,p})|^j]dr\right)\ind_{\{l \le p-1\}}.
  \end{align*}
  Using \eqref{eq32} and \eqref{eq31b} leads to 
 \begin{align*}
   &  \sum_{1\le l \le m} \sum_{\mathbf{i}_l \in \{0,1\}^l} \operatorname*{ess \,sup}_{t_1, \cdots, t_l} \int_{\tilde{t}_l}^T 
\e[|D^{\mathbf{i}_l}_{t_1, \cdots, t_l} Z^{q+1,p}_r|^j]dr \nonumber \\
&\le C(p,j)\left( \|\xi\|^j_{p,1}+\sum_{l=1}^p C \left (\sum_{k=1}^{l}
      \|\partial^k_{sp}f\|^j_{\infty} \right) \|(Y^{q,p},Z^{q,p},U^{q,p})\|^{lj}_{l,lj}\right).
  \end{align*}
  The same type of result holds for $\int_{\tilde{t}_l}^T \e[|D^{\mathbf{i}_l}_{t_1, \cdots, t_l} U^{q+1,p}_r|^j]dr $.
  Combining these results with \eqref{eq34} gives
  $$\|(Y^{q+1,p},Z^{q+1,p},U^{q+1,p})\|^j_{m,j} \le \!
  C(p,j)\left(\|\xi\|^j_{p,1}+\sum_{l=1}^{m\vee p} C\left (\sum_{k=1}^l \|\partial^k_{sp}f\|^j_{\infty} \right ) \|(Y^{q,p},Z^{q,p},U^{q,p})\|^{lj}_{l,lj}\right).
  $$
  Iterating this inequality yields the result.
  
\subsection{Proof of Lemma \ref{basis-truncation}}\label{sec:proof_basis-truncation}

We will prove the assertion by induction  in $p \in \nset.$ Since $ (\mathcal{C}_0^N)(F) =(\mathcal{C}_0)(F)$
Lemma \ref{basis-truncation}  holds for $p=0.$ Assume that  for $p \in \nset^*$
\[
   \e|   (\mathcal{C}_{p-1}^N -\mathcal{C}_{p-1})(F)|^2 \le   \sum_{j=1}^{p-1}(K_{j}^F)^2 \bigg ( \frac{T}{N} \bigg)^{2 \beta_F} \sum_{i=1}^{p-1} i^2 \frac{T^i}{i!}.
   \]
 By using \eqref{eq5} and \eqref{eq51}, we get
\[
 (\mathcal{C}_p^N -\mathcal{C}_p)(F) =(\mathcal{C}_{p-1}^N -\mathcal{C}_{p-1})(F) +(P_p^N - P_p)(F).
 \]
Then, it suffices to show that 
\[\e |(P_p^N - P_p)(F)|^2 \le  (K_p^F)^2  \left ( \frac{T}{N} \right )^{2 \beta_F} p^2 \frac{T^p}{p!}.\]
We have $ P_p(F) = I_p(g_p)$ where we will assume that $g_p$ is symmetric. It holds  
\[
   P^N_p(F) = I_p(g_p^N) 
\]
with 
\begin{align*}
     g_p^N((t_1,i_1),\cdots,(t_p,i_p)) =&  \sum_{{\bf k}_p \in\{1,\cdots, N\}^p} \langle g_p((\cdot,i_1),\cdots,(\cdot,i_p)),  e[k_1,...,k_p]\rangle_{L^2([0,T]^p)} \\
     & \quad \quad \quad \quad \times  e[k_1,...,k_p](t_1,\cdots, t_p).
     \end{align*}

Then $g_p^N$ is constant w.r.t. $(t_1,\cdots,t_p) \in \Lambda_{ {\bf k}_p} := \Lambda_{k_1} \times \cdots \times \Lambda_{k_p} $
with $\Lambda_i := ]\ov{t}_{i-1},\ov{t}_i]$ since $e[k_1,...,k_p] =   h^{-\frac{p}{2}} \ind_{\Lambda_{ {\bf k}_p} }.  $
We have by \eqref{isometry-formula}, \eqref{kernel-Malliavin-rel} and
\eqref{eq31} that

\begin{align*}
&  \e|  (P_p^N - P_p)(F)|^2 \\&=  \e|  I_p(g_p^N)     - I_p(g_p)|^2 \\
  &= \sum_{{\bf k}_p} \sum_{{\bf i}_p}  \kappa^{|{\bf i}_p  |}  p!   \|  h^{-\frac{p}{2}} \langle g_p((\cdot,i_1),\cdots,(\cdot,i_p)),  e[k_1,...,k_p]\rangle_{L^2([0,T]^p)}  - g_p((\cdot,i_1),\cdots,(\cdot,i_p) )
  \|^2_{L_2(\Lambda_{ {\bf k}_p})} \\
   &= \sum_{{\bf k}_p} \sum_{{\bf i}_p}  \kappa^{|{\bf i}_p  |}  p!  \left \|  h^{-p}\!\!\!   \int_{\Lambda_{ {\bf k}_p}}  g_p((s_1,i_1),\cdots,(s_p,i_p))  - g_p((\cdot,i_1),\cdots,(\cdot,i_p)) ds_1 ... d s_p
   \right \|^2_{L_2(\Lambda_{ {\bf k}_p})} \\
   &\le \sum_{{\bf k}_p{\bf i}_p}  \kappa^{|{\bf i}_p  |}  p!
   h^{-2p}\!\!\!  \int_{\Lambda_{\mathbf{k}_p}}\!\!\!  \left(\int_{\Lambda_{ {\bf k}_p}}
   \!\!\!   |g_p((s_1,i_1),\cdots,(s_p,i_p))  - g_p((t_1,i_1),\cdots,(t_p,i_p))|
     ds_1 ... d s_p\!\right)^2 \!\!\!  dt_1 \cdots dt_p\\
    &\le \sum_{{\bf k}_p} \sum_{{\bf i}_p}  \kappa^{|{\bf i}_p  |}  \frac{1}{p!}
   h^{-2p}\int_{\Lambda_{\mathbf{k}_p}} \left(\int_{\Lambda_{ {\bf k}_p}}
     K^F_p(|t_1-s_1|^{\beta_{F}}+\cdots+|t_p-s_p|^{\beta_{F}})
     ds_1 ... d s_p\right)^2 dt_1 \cdots dt_p\\
   &\le  (K^F_p)^2T^p(1+ \kappa)^p p^2  \frac{1}{p!} \big ( \frac{T}{N} \big )^{2
   \beta_F}.
\end{align*}

  \subsection{Proof of Lemma \ref{lem50}}\label{proof_lem50}

 We will show that if 
\begin{align}  \label{induction-assumption-onYZU}
(Y_t^{q,p},Z_t^{q,p},U_t^{q,p})  \text{  satisfies }  \mathcal{H}_p  \text{ for a.e. } t \in [0,T]
\end{align} 
(with  $\beta_{I_{q,p}}=\frac{1}{2}\wedge \beta_{\xi}$)
then also  $I_{q,p}=\int_0^T f(s,Y_s^{q,p},Z_s^{q,p},U_s^{q,p})ds$ does satisfy $\mathcal{H}_p$. 
As $I_{0,p}$ is constant, it satisfies $\mathcal{H}_p$.   For $q \ge 1$
we will use the notation  $D_{\mathbf{t}}^{\alpha(1:i-1)} \Delta^{\alpha_i}_i
D_{\mathbf{s}}^{\alpha(i+1: r)} F   :=D^{(\alpha_1,\cdots,\alpha_{i-1})}_{t_1,\cdots,t_{i-1}}
  (D_{t_i,s_{i+1},\cdots,s_r}^{(\alpha_i,\cdots,\alpha_r)}F -
  D_{s_i,s_{i+1},\cdots,s_r}^{(\alpha_i,\cdots,\alpha_r)}F)$
 and  prove that for $1 \le r\le p$
 $$\e |D^{\alpha(1:i-1)}_{\mathbf{t}} \Delta^{\alpha_i}_i D^{\alpha(i+1:r)}_{\mathbf{s}} I_{q,p}|^j \le
K_r(j) |t_i-s_i|^{j\beta_{I_{q,p}}}$$ 
(that $\mathcal{H}^1_p$ holds for $I_{q,p}$ is clear).
Setting  $\mathbf{ts}_{-i} :=\max\{t_1,\cdots,t_{i-1},s_{i+1},\cdots,s_r\}$ and $\theta_u^{q,p} :=(u,Y_u^{q,p},Z_u^{q,p},$ $U_u^{q,p})$ we have
\begin{align} \label{DIqp}
 D^{\alpha(1:i-1)}_{\mathbf{t}} \Delta^{\alpha_i}_i
D^{\alpha(i+1:r)}_{\mathbf{s}} I_{q,p} &= \int_{t_i \vee s_i \vee \mathbf{ts}_{-i}}^TD^{\alpha(1:i-1)}_{\mathbf{t}} 
    \Delta^{\alpha_i}_i D^{\alpha(i+1:r)}_{\mathbf{s}}  f(\theta_u^{q,p} )du  \nonumber \\
& \pm \int_{(t_i \wedge s_i) \vee \mathbf{ts}_{-i}}^{t_i \vee s_i \vee \mathbf{ts}_{-i}} D^{\alpha(1:i-1)}_{\mathbf{t}} 
    D^{\alpha_i}_{t_i\wedge s_i} D^{\alpha(i+1:r)}_{\mathbf{s}}  f(\theta_u^{q,p} )du,
\end{align}
where $\pm = -$ for $t_i > s_i,$  and $\pm = +$ for the opposite case.
From the proof of Lemma \ref{lem5} we know that  $|D^{\alpha(1:i-1)}_{\mathbf{t}} 
    D^{\alpha_i}_{t_i\wedge s_i} D^{\alpha(i+1:r)}_{\mathbf{s}}  f(\theta_u^{q,p} )|$ is bounded
by a sum of
terms of type
\begin{align*}
\left (\sum_{j=1}^r \|\partial_{sp}^j f\|_{\infty} \right )
  |  D^{\mathbf{k}_0}_{\mathbf{t}_0} Y_u^{q,p} | \,
|D^{\mathbf{k}_1}_{\mathbf{t}_1} Z_u^{q,p}| \,|
D^{\mathbf{k}_2}_{\mathbf{t}_2} U_u^{q,p}| ,
\end{align*}
where $\mathbf{k}_j \in \{0,1\}^{l_j}$ are
vectors of size  $l_j$ with $l_0+l_1+l_2 \le r$ and $\mathbf{t}_j$ are sub vectors of $\{t_1,...,t_{i-1}, t_i\wedge s_i, s_{i+1},...,s_r\}.$
Hölder's inequality and Lemma \ref{lem5b} give 
\[\e \left |\int_{(t_i \wedge s_i) \vee \mathbf{ts}_{-i}}^{t_i \vee s_i \vee \mathbf{ts}_{-i}} D^{\alpha(1:i-1)}_{\mathbf{t}} 
    D^{\alpha_i}_{t_i\wedge s_i} D^{\alpha(i+1:r)}_{\mathbf{s}}  f(\theta_u^{q,p} )du \right |^j \le C(p,j,\|\xi\|_{p,1},(\|\partial_{sp}^k f\|_{\infty})_{k\le p})  |t_i-s_i|^{\frac{j}{2}}.\]
    
For the first term on the r.h.s. of \eqref{DIqp} we notice that
\[
 \int_{t_i \vee s_i \vee \mathbf{ts}_{-i}}^T   |D^{\alpha(1:i-1)}_{\mathbf{t}} 
    \Delta^{\alpha_i}_i D^{\alpha(i+1:r)}_{\mathbf{s}}  f(\theta_u^{q,p} ) |du  
\]
is bounded by a sum of terms of type
$$\int_{\mathbf{ts}_{-i}}^T  \bigg ( \sum_{j=1}^r \|\partial_{sp}^j f\|_{\infty}  \bigg )  |  D^{\mathbf{k}_0}_{\mathbf{t}_0} \Phi_u^{q,p} | \,
|D^{\mathbf{k}_1}_{\mathbf{t}_1} \Psi_u^{q,p}| \, |
D^{\mathbf{k}_2}_{\mathbf{t}_2}  \Delta^{\alpha_i}_iD^{\mathbf{k}_3}_{\mathbf{t}_3} 
 \Gamma_u^{q,p} |du,$$ 
where   $(\Phi_u^{q,p},\Psi_u^{q,p}, \Gamma_u^{q,p} )$ stands for a permutation of  $\{Y_u^{q,p},Z_u^{q,p}, U_u^{q,p} \},$ and  $\mathbf{k}_j \in \{0,1\}^{l_j}$ are
vectors of size  $l_0,l_1, l_2+l_3+1 \le r$ while the $\mathbf{t}_j$ denote the appropriate sub vectors of $\{t_1,...,t_{i-1}, t_i, s_i, s_{i+1},...,s_r\}.$
 
By Hölder's inequality and  assumption \eqref{induction-assumption-onYZU} we conclude that 
\[
\e \left  | \int_{t_i \vee s_i \vee \mathbf{ts}_{-i}}^TD^{\alpha(1:i-1)}_{\mathbf{t}} 
    \Delta^{\alpha_i}_i D^{\alpha(i+1:r)}_{\mathbf{s}}  f(\theta_u^{q,p} )du \right |^j \le  K_{r}(j) |t_i-s_i|^{j\beta_{I_{q,p}}}.
\]

We finish the proof of  Lemma \ref{lem50}  by arguing that   assumption   \eqref{induction-assumption-onYZU} 
holding for true for a certain $q,$  implies  it   for $q+1:$  
We want to use  \eqref{eq:Y_qp} and \eqref{eq:ZU_qp} and therefore we first notice that in the same way as above for  $I_{q,p}$ one can show that  
\eqref{induction-assumption-onYZU} implies that  
$$ \int_0^t  f\left(s,Y^{q,p}_s,Z^{q,p}_s,U^{q,p}_s\right)ds
\text{ \,\,\, satisfies  \,\,\,} \mathcal{H}_p.$$ It is also clear that  satisfying  $\mathcal{H}_p$ is stable
  with respect to linear combination and taking the conditional expectation $\e_t.$  What  we still need to check is whether satisfying  $\mathcal{H}_p$  
  is also stable with respect to the truncation  $\cc_p.$ For this, let us assume that $F=\sum_{n=0}^{\infty} I_n(g_n)$   satisfies  $\mathcal{H}_p.$ 
Following the proof of Lemma \ref{lem5b}, we have
\begin{align*}
  &D^{\alpha(1:i-1)}_{\mathbf{t}} \Delta^{\alpha_i}_i D^{\alpha(i+1:r)}_{\mathbf{s}} F\\
   &=\sum_{n=r}^{\infty}n(n-1)\cdots(n-r+1)I_{n-r}(g_n(*,z_1,\cdots,z_i,z'_{i+1},\cdots,z'_r) \\
   & \quad\quad  \quad\quad \quad\quad -g_n(*,z_1,\cdots, z_{i-1},z'_i\cdots,z'_r)), 
 \end{align*}
where $z_j=(t_j, i_j)$ and  $z'_j=(s_j, i_j).$ Like in \eqref{E-of-product-of-chaos} we get 
\begin{align*}
 & \e[|\cc_{p-r} (D^{\alpha(1:i-1)}_{\mathbf{t}} \Delta^{\alpha_i}_i D^{\alpha(i+1:r)}_{\mathbf{s}} F)|^j]  \\
&  \le  C(p,j,T)  \sum_{n_1,\cdots,n_j=0}^{p-r} \frac{(n_1+r)!}{n_1!}\cdots \frac{(n_j+r)!}{n_j!}  \\  & \,\,\times  \prod_{i=1}^j
  \|g_{n_i+r}(*,z_1,\cdots,z_i,z'_{i+1},\cdots,z'_r) -g_{n_i+r}(*,z_1,\cdots, z_{i-1},z'_i\cdots,z'_r))  \|_{\infty}  \\
  & \le  C(p,j,T)  ( K_ p^F (1) |t_i-s_i|^{\beta_{F}} )^j,
 \end{align*}
where we used that 
\begin{align*}
 (n_i+r)!\|g_{n_i+r}(*,z_1,\cdots,z_i,z'_{i+1},\cdots,z'_r) -&g_{n_i+r}(*,z_1,\cdots, z_{i-1},z'_i\cdots,z'_r))  \|_{\infty} \\ &\le  K_ p^F(1) |t_i-s_i|^{\beta_{F}}.
\end{align*}

\subsection{Proof of Lemma \ref{lem7}}\label{sec:proof_lem7}
Using the definitions \eqref{eq:chaos_dec} and \eqref{chaos_dec_MC} leads to
\begin{align*}
  (\cc_p^N-\cc_p^{N,M})(F)=d_0-\hat{d_0}+\sum_{k=1}^p \sum_{|{\bf n}|=k}
  (d^{\bf n}_k-\hat{d^{\bf n}_k})\prod_{i=1}^N K_{n_i^B}(G_i)C_{n_i^P}(Q_i,\kappa h).
\end{align*}
Since $\hat{d^{\bf n}_k}$ is independent of $(G_i,Q_i)_{1\le i \le N}$
\begin{align*}
  \e(|(\cc_p^N-\cc_p^{N,M})(F)|^2)=\e(|d_0-\hat{d_0}|^2)+\sum_{k=1}^p
  \sum_{|{\bf n}|=k}\frac{({\bf n}^P)!(\kappa h)^{|{\bf n}^P|}}{({\bf n}^B)!}\e(|d^{\bf n}_k-\hat{d^{\bf n}_k}|^2).
\end{align*}
The definition of the coefficients $d_0$ and $d^{\bf n}_k$ given in
\eqref{eq:coef_chaos_dec} leads to
\begin{align*}
  \e(|(\cc_p^N-\cc_p^{N,M})(F)|^2)=\V(\hat{d_0})+\sum_{k=1}^p
  \sum_{|{\bf n}|=k}\frac{({\bf n}^P)!(\kappa h)^{|{\bf n}^P|}}{({\bf n}^B)!}\V(\hat{d^{\bf n}_k}).
\end{align*}
Using the definition of $\hat{d^{\bf n}_k}$ (see \eqref{d_hat}) leads to the first result.
To get the second result, we write
$\cc_p^{N,M}(F)=(\cc_p^{N,M}-\cc_p^{N})(F)+\cc_p^N(F)$. Since
$\e\left((\cc_p^{N,M}-\cc_p^{N})(F)\cc_p^N(F)\right)=0$, we get
\begin{align*}
  \e(|\cc_p^{N,M}(F)|^2)=\e(|(\cc_p^{N,M}-\cc_p^{N})(F)|^2)+\e(|\cc_p^N(F)|^2).
\end{align*}
Lemma \ref{lem9} ends the proof.

 \subsection{The product of  two multiple integrals}\label{LeeShihformula}

For the convenience of the reader, we cite here  \cite[Theorem 3.6]{LS_04} from Lee and Shih adapted 
to our  simple situation where the multiple  integrals  $I_k(g_k)$ are built using the process $B+\tilde N$ like in 
Subsection \ref{multiple}. For this, we first  introduce the 'contraction and identification operator' 
$ \otimes_a^b.$ For symmetric functions  $g_k \in (L^2)^{\otimes k}( \lambda \otimes(\delta_0 + \kappa \delta_1))$ and 
$g_m \in (L^2)^{\otimes m}( \lambda \otimes(\delta_0 + \kappa \delta_1))$ we define 
the function $ g_k   \otimes_a^b g_m:  ([0,T] \times \{0,1\} )^{k-a-b} \times  ([0,T] \times \{0,1\} )^{m-a-b} \times ([0,T] \times \{0,1\} )^{b} \to \rset $
by
\begin{align}\label{contraction-identification}
(g_k  \otimes_a^b g_m)({\bf x }, {\bf y }, {\bf z })  =  \int_{([0,T] \times \{0,1\} )^a} g_k ({\bf x }, {\bf z }, {\bf w }) g_m ({\bf w }, {\bf z }, {\bf y}) d[ \lambda \otimes(\delta_0 + \kappa \delta_1)]^{\otimes a}({\bf w })
\end{align}
for  $({\bf x }, {\bf y }, {\bf z })  \in  ([0,T] \times \{0,1\} )^{k-a-b} \times  ([0,T] \times \{0,1\} )^{m-a-b} \times ([0,T] \times \{0,1\} )^{b} .$
\begin{thm}  If  $g_k \in (L^2)^{\otimes k}( \lambda \otimes(\delta_0 + \kappa \delta_1))$ and 
$g_m \in (L^2)^{\otimes m}( \lambda \otimes(\delta_0 + \kappa \delta_1))$ are symmetric functions such that 
 $|g_k| \otimes_a^b |g_m|$  is in  $(L^2)^{\otimes (k+m-2a-b)}( \lambda \otimes(\delta_0 + \kappa \delta_1)) ,$ then
\[
I_k (g_k)  I_m(g_m)
= \sum_{a=0}^{k \wedge m} \sum_{b=0}^{k \wedge m-a}   a!b!  \binom{k}{a} \binom{m}{a} \binom{k-a}{b} \binom{m-a}{b}   I_{k+m-2a-b}( g_k   \otimes_a^b g_m).
\]
\end{thm}
An immediate consequence is that if $g_k$ and  $g_m$   have disjoint support, then 
$I_k (g_k)  I_m(g_m) =I_{k+m}(g_k \otimes g_m).$ 




\bigskip

{\large \bf References}
 \bibliographystyle{abbrv} 
\bibliography{ref}





\end{document}